\documentclass[twoside]{article}
\long\def\symbolfootnote[#1]#2{\begingroup%
\def\thefootnote{\fnsymbol{footnote}}\footnote[#1]{#2}\endgroup}

\usepackage{amsfonts, amssymb, amsmath, amsthm}
\usepackage[colorlinks=true, pdfstartview=FitV, linkcolor=blue, 
            citecolor=blue, urlcolor=blue]{hyperref}
\usepackage[usenames]{color}
\definecolor{Red}{rgb}{0.7,0,0.1}
\definecolor{Green}{rgb}{0,0.7,0}
\usepackage{accents}
\usepackage{comment}

\title{Global Existence and Regularity for the 3D Stochastic Primitive
  Equations of the Ocean and Atmosphere with Multiplicative White Noise\footnote{
  {\bf To appear in Nonlinearity.}
  }}
\author{A. Debussche$^{\sharp}$,  
N. Glatt-Holtz$^\flat$, R. Temam$^\flat$ and 
M. Ziane$^{\natural}$}

\date{}

\numberwithin{equation}{section}

\newtheorem{Thm}{Theorem}[section]
\newtheorem{Lem}{Lemma}[section]
\newtheorem{Prop}{Proposition}[section]
\newtheorem{Cor}{Corollary}[section]

\newtheorem{Def}{Definition}[section]
\newtheorem{Rmk}{Remark}[section]

\newcommand{\pd}[1]{\partial_{#1}}
\newcommand{\indFn}[1]{1 \! \! 1_{#1}}
\newcommand{\E}{\mathbb{E}}
\newcommand{\Prb}{\mathbb{P}}

\newcommand{\dM}{d\mathcal{M}}
\newcommand{\dMo}{d\mathcal{M}_0}

\newcommand{\tauz}{\tau_K^{\pd{z}\mathbf{v}}}
\newcommand{\tauT}{\tau_K^{T,S}}
\newcommand{\tauM}{\sigma_K}

\includecomment{commentCoAuthor}

\includecomment{commentToDo}

\begin{document}


\maketitle

\vskip-4mm

\centerline{\footnotesize{\it $^{\sharp}$IRMAR-UMR 6625, ENS Cachan Bretagne}}
\vskip-1mm
\centerline{\footnotesize{\it 35170 Bruz, France}}

\vskip2mm

\centerline{\footnotesize{\it $^\flat$The Institute for Scientific Computing and Applied Mathematics} }
\vskip-1mm
\centerline{\footnotesize{\it Indiana University, Bloomington, IN 47405, USA}}

\vskip2mm

\centerline{\footnotesize{\it $^{\natural}$Department of Mathematics, University of Southern California}}
\vskip-1mm
\centerline{\footnotesize{\it Los Angeles, CA 90089, USA}}

\begin{center}
\large
\date{\today}
\end{center}

\vskip4mm

\begin{abstract}
The Primitive Equations are a basic model
in the study of large scale Oceanic and Atmospheric dynamics.
These systems form the analytical core of the most advanced General
Circulation Models.
For this reason and due to their challenging nonlinear and anisotropic structure
the Primitive Equations have recently received considerable attention from the 
mathematical community.

On the other hand, in view of the complex multi-scale nature of the earth's climate system,
many uncertainties appear that should be accounted for in the 
basic dynamical models of atmospheric and oceanic processes.   
In the climate community
stochastic methods have come into extensive use in this connection. 
For this reason there 
has appeared a need to further develop the foundations of nonlinear 
stochastic partial differential equations in connection with the 
Primitive Equations and more generally. 

In this work we study a stochastic version of the Primitive 
Equations.  We establish the global existence and uniqueness of strong, pathwise solutions for
  these equations in dimension 3 for the case of a nonlinear
  multiplicative noise.  The proof makes use of anisotropic estimates,
  $L^{p}_{t}L^{q}_{x}$ estimates on the pressure and stopping time 
  arguments.\\ \\ \\  \\ \\ \\
  
\end{abstract}

{\noindent \small 
  }

\newpage
\section{Introduction}
\label{sec:introduction}

The Primitive Equations (PEs) are widely considered to be a fundamental model
in the study of large scale oceanic and atmospheric dynamics.
These systems form the analytical core of the most advanced 
general circulation models for the atmosphere (AGCMs) or oceans
(OGCMs) or the coupled oceanic-atmospheric system (GCMs).
Moreover, beyond their considerable significance in applications, the 
PEs have generated much interest from the mathematics community 
due to their rich nonlinear, nonlocal character and their anisotropic structure.

The physical derivation of the Primitive Equations 
goes back to the early 20th century \cite{ Bjerknes1, Richardson1}. 
It is based on a
scale analysis that accounts for the relatively constant density 
of the ocean (in the atmosphere the density follows a linear profile)
and the contrast between the vertical and horizontal scales (on
the order of several kilometers vs. thousands of kilometers).  We refer
the reader to e.g. \cite{Pedlosky} (or \cite{Richardson1})
for further physical background for the (deterministic)
PEs.

Hence the Primitive Equations express
very fundamental laws of physics and one may
wonder what the motivations are for introducing uncertainty
in ``exact'' model equations.  The introduction
of stochastic processes in weather and climate prediction
is aimed at accounting for a number of uncertainties
and errors:
\begin{itemize}
\item[1.] These ``exact'' models are numerically
intractable;  they cannot be fully solved with present super
computers (and will not be for any foreseeable future).
For this reason some sort of statistical averages are 
needed corresponding to the parameterization
of the effects of the small unresolved scales.
Indeed when the full equations are averaged it is
common to introduce volumic stochastic
forcing terms as a partial closure. This is reflected
in a growing physics literature on `stochastic parameterization'.
See, for example, \cite{Rose1, 
LeslieQuarini1, MasonThomson, BernerShuttsLeutbecherPalmer, 
ZidikheriFrederiksen}.
\item[2.] The physics of these ``exact'' models is
in fact far from being fully understood.  
For the atmosphere the most dramatic example of
physical uncertainty is due to the radiation properties of
the air (clouds) which produce the source term called $F_T$
in the atmospheric analogue of the equation 
\eqref{eq:diffEqnTempPE} below,
that is the energy balance equation. Nowadays
these uncertainties due to the radiation properties
of the clouds are considered to be the most severe
source of uncertainty in weather and climate
modeling. Note that while our presentation focuses on the
equations of the ocean (which are mathematically
slightly simpler than the corresponding atmospheric
or coupled oceanic-atmospheric equations) our general
results and methods in this work apply also to these 
systems.\footnote{Physical uncertainties on the boundary
of the oceanic or coupled oceanic-atmospheric system
are also present and are responsible for another source of
stochasticity in these models.  The boundary conditions
\eqref{eq:3dPEPhysicalBoundaryCondTop} on $\Gamma_i$
(which is physically the sea surface-air interface) are a 
simplified version of the more realistic boundary conditions
\begin{equation}\label{eq:hardBc}
  \alpha_{\mathbf{v}} (\mathbf{v} - \mathbf{v}_*) + \pd{z} \mathbf{v} = g_{\mathbf{v}},
  \quad \alpha_{T} (T - T_*)+ \pd{z} T = g_{T}.
\end{equation}
In the form \eqref{eq:hardBc} the boundary condition expresses
two fundamental aspects of oceanic-atmospheric interaction 
namely the driving force of the wind (the term $g_\mathbf{v}$)
and the heating or cooling of the air by the ocean (the term
$g_T$).  These functions $g_\mathbf{v}$ and $g_T$ are not
well known and are estimated by modelers through very rough averages.   Thus
the uncertainties of these functions lead to stochastic PDEs with
white noise on the boundary a subject which we intend to pursue
in future work.
}
\end{itemize}
Further and related physical 
background in this connection is
the subject of \cite{GlattHoltzTemamTribbia1}.

With this backdrop in mind we study in this work a stochastic 
version of the primitive
equations and establish the global existence and uniqueness of solutions
for a nonlinear, multiplicative white noise.    While the mathematical theory is
now extensive in the deterministic case the stochastic primitive equations are
only receiving attention very recently.   We next review this mathematical 
background and then conclude the introduction by describing our results 
and outlining the analysis carried out below.

\subsubsection*{Previous Mathematical Work}

To the best of our knowledge the mathematical 
study of the Primitive Equations started
in the early 1990s with a series of works
\cite{LionsTemamWang1, LionsTemamWang2, LionsTemamWang3}
establishing the existence for all time of weak solutions of these equations.
Subsequent articles improved these results 
and obtained existence and uniqueness of more regular 
(strong) solutions, very similar to the results available
for the incompressible Navier-Stokes equations 
(see e.g. \cite{Guillen-GonzalezMasmoudiRodriguezBellido1, HuTemamZiane2003, PetcuTemamZiane}).
Recently, taking advantage of the fact that the pressure is 
essentially two-dimensional in the PEs (unlike
the Navier-Stokes equations) global results 
for the existence of strong solutions of the full
three dimensional PEs were established
in \cite{CaoTiti} and independently in \cite{Kobelkov, Kobelkov2007}.   In subsequent
work, \cite{ZianeKukavica}, a different proof was developed which
allows one to treat non-rectangular domains as well as different, physically realistic, 
boundary conditions.   All of these 
works make essential use of the fact that the pressure terms appearing in 
the PEs can be shown to be essentially independent 
of the vertical variable $z$.   In the present work we follow an
approach closer to that of \cite{ZianeKukavica} in that we estimate
the pressure directly via earlier results for the evolution Stokes equation as
in \cite{SohrVonWahl1}. 
In any case the deterministic mathematical theory for the Primitive 
equations has now reached an advanced stage and we refer the reader
to the survey articles \cite{PetcuTemamZiane, RousseauTemamTribbia}
for further references and background.

Notwithstanding these extensive results in the deterministic case, the theory for
the stochastic Primitive Equations remains underdeveloped.
A two dimensional version of the PEs has been studied 
in a simplified form in \cite{EwaldPetcuTemam, GlattHoltzZiane1} 
and more recently in \cite{GlattHoltzTemam1,GlattHoltzTemam2}
in the greater generality of physically relevant boundary conditions
and nonlinear multiplicative noise.
While the full three dimensional system has been studied in
\cite{GuoHuang} following the methods in \cite{CaoTiti},
this work covers only the case of additive noise.  In this case
the PEs can be directly studied pathwise via a classical change of variables.
In this and a companion work \cite{DebusscheGlattHoltzTemam1}
devoted to the local existence of solutions we depart from \cite{GuoHuang} 
and develop different methods
which allow us, in particular, to consider a nonlinear multiplicative 
forcing structure.  The present work may therefore be seen as the continuation of  
\cite{DebusscheGlattHoltzTemam1} which takes us from
the local to the global existence of solutions.

In contrast to the primitive equations, the theory of the related stochastic
Navier-Stokes equations has undergone substantial developments
dating back to the 1970's with the initial work \cite{BensoussanTemam}.
In this literature, as for the wider literature on stochastic PDEs (see \cite{ZabczykDaPrato1})
two principal notions of solutions have been developed.   
On the one hand Martingale (or probabilistically weak) solutions treat
the stochastic elements in the problem as an unknown. 
In this case, one typically establishes compactness in the space of probability laws associated to
solutions. See e.g. \cite{Viot1, Cruzeiro1, CapinskiGatarek, FlandoliGatarek1,MikuleviciusRozovskii2}.
On the other hand Pathwise (or probabilistically strong) solutions consider the driving
noise as being fixed in advance.  See, for example, \cite{ZabczykDaPrato2,
BensoussanFrehse, Breckner, BrzezniakPeszat, FlandoliRomito1, DaPratoDebussche, 
MikuleviciusRozovskii4, Shirikyan1, FlandoliRomito, GlattHoltzZiane2}.
In this and the companion work \cite{DebusscheGlattHoltzTemam1}
we concentrate on this latter notion.


\subsubsection*{The Governing Equations}
Having described the background and motivations for this article, we now 
outline its content in more detail.
The stochastic version of the Primitive Equations that we study 
below takes the form:
\begin{subequations}\label{eq:PE3DBasic}
  \begin{gather}
    \begin{split}
    \pd{t} \mathbf{v} 
    + (\mathbf{v} \cdot \nabla)\mathbf{v} + w \pd{z}\mathbf{v}
    + \frac{1}{\rho_0} \nabla p 
    +& f \mathbf{k} \times \mathbf{v} 
    - \mu_{\mathbf{v}} \Delta \mathbf{v} 
    - \nu_{\mathbf{v}} \pd{zz} \mathbf{v}
    = F_{\mathbf{v}} + \sigma_{\mathbf{v}}(\mathbf{v},T,S) \dot{W}_1,
    \end{split}
    \label{eq:MomentumPE}\\
    \pd{z} p = - \rho g,
    \label{eq:HydroStaticPE}\\
    \nabla \cdot \mathbf{v} + \pd{z} w = 0
    \label{eq:divFreeTypeCondPE}\\
    \pd{t} T + (\mathbf{v}\cdot \nabla) T
             + w \pd{z} T
             - \mu_{T} \Delta T
             - \nu_{T} \pd{zz} T
             = F_{T} + \sigma_{T}(\mathbf{v},T,S) \dot{W}_2,
    \label{eq:diffEqnTempPE}\\
    \pd{t} S + (\mathbf{v}\cdot \nabla) S
             + w \pd{z} S
             - \mu_{S} \Delta S
             - \nu_{S} \pd{zz} S
             = F_{S} + \sigma_{S}(\mathbf{v},T,S) \dot{W}_3,
    \label{eq:diffEqnSaltPE}\\
    \rho = \rho_0 ( 1 + \beta_T( T - T_r) + \beta_S(S- S_r)).
    \label{eq:linearDensityDependence}
  \end{gather}
\end{subequations}
Here, $U := (\mathbf{v},T,S)= (u,v, T, S)$, $p$, $\rho$ represent the horizontal
velocity, temperature, salinity, pressure and density of the fluid under
consideration; $\mu_{\mathbf{v}}$, $\nu_{\mathbf{v}}$, $\mu_{T}$, $\nu_{T}$,
$\mu_{S}$, $\nu_{S}$
are (possibly anisotropic) coefficients of the eddy and molecular
viscosity and the heat and saline diffusivity respectively; $f$ is the Coriolis
parameter appearing in the antisymmetric term in \eqref{eq:MomentumPE}
and accounts for the earth's rotation in the momentum equations.
The third component of the velocity field, $w$ is a `diagnostic variable' in 
that it is determined directly from $\mathbf{v}$, the components
of the horizontal velocity field (see \eqref{eq:divFreeTypeCondPEInt}, below). 
The evolution equations \eqref{eq:PE3DBasic} occurs for $(x_1,x_2,z)$
ranging over a cylindrical domain
$\mathcal{M} = \mathcal{M}_0 \times (-h, 0)$;  $\mathcal{M}_0$ is
an open bounded subset of $\mathbb{R}^2$ with smooth boundary $\partial \mathcal{M}_0$.  
Note that $\nabla = (\pd{1}, \pd{2})$ where $\pd{1}$,
$\pd{2}$ are the partial derivatives in the (horizontal) $x_1, x_2$ directions;
$\Delta = \pd{1}^2 + \pd{2}^2$ is the horizontal Laplace operator.

The stochastic terms are driven by independent Gaussian white noise 
processes $\dot{W}_{j}$ which are formally delta correlated in time.  
The stochastic terms may be written formally in the expansion
\begin{equation}\label{eq:expStochFormal}
        \left(
      \begin{array}{c}
    \sigma_{\mathbf{v}}(U) \dot{W}_1(t,x)\\
    \sigma_{T}(U)\dot{W}_2(t,x)\\
    \sigma_{S}(U)\dot{W}_3(t,x)\\
  \end{array}
      \right)
      =        \sum_{k \geq 1}  \left(
      \begin{array}{c}
    \sigma_{\mathbf{v}}^k(U)(x,t) \dot{W}_1^k(t)\\
    \sigma_{T}^k(U)(x,t)\dot{W}_2^k(t)\\
    \sigma_{S}^k(U)(x,t)\dot{W}_3^k(t)\\
  \end{array}
      \right),
\end{equation}
where the elements $\dot{W}_j^k$ are independent white (in time) noise
processes.
We understand  \eqref{eq:PE3DBasic} in the It\={o} sense but
the classical correspondence between the It\={o} and Stratonovich systems
would allow one to treat both situations with the analysis herein.  We 
recall the basic mathematical definitions and give precise conditions
on the operators $\sigma_{\mathbf{v}}$, $\sigma_T$, $\sigma_S$ below.

The boundary $\partial \mathcal{M}$
is partitioned into the top $\Gamma_i = \mathcal{M}_0 \times \{0\}$, the
bottom $\Gamma_b= \mathcal{M}_0 \times \{-h\}$ and the sides $\Gamma_l
= \partial \mathcal{M}_0 \times (-h, 0)$. 
We denote by $\mathbf{n}_H$ the outward
unit normal to $\partial \mathcal{M}_0$.
We prescribe
the following boundary conditions:
\begin{equation}\label{eq:3dPEPhysicalBoundaryCondTop}
  \begin{split}
    \pd{z}\mathbf{v} = 0, \quad w = 0,
    \quad
    \pd{z} T = 0,
    \quad
    \pd{z} S = 0,
  \end{split}
\end{equation}
on $\Gamma_i$.  At the bottom $\Gamma_b$ we take
\begin{equation}\label{eq:3dPEPhysicalBoundaryCondBottom}
  \begin{split}
  \pd{z} \mathbf{v} = 0, \quad w = 0, \quad
  \pd{z} T = 0, \quad \pd{z} S = 0.\\    
  \end{split}
\end{equation}
Finally for the lateral boundary $\Gamma_l$
\begin{equation}\label{eq:3dPEPhysicalBoundaryCondSide}
  \mathbf{v} = 0, \quad \pd{\mathbf{n}_H}T = 0, \quad \pd{\mathbf{n}_H}S = 0.
\end{equation}
The equations and boundary conditions \eqref{eq:PE3DBasic},
\eqref{eq:3dPEPhysicalBoundaryCondTop},
\eqref{eq:3dPEPhysicalBoundaryCondBottom},
\eqref{eq:3dPEPhysicalBoundaryCondSide} are supplemented 
by initial conditions for $\mathbf{v}$, $T$ 
and $S$, that is
\begin{equation}\label{eq:basicInitialCond}
   \mathbf{v} = \mathbf{v}_0, \quad
   T = T_0, \quad
   S = S_0, \quad
   \textrm{ at } t = 0.
\end{equation}

The Cauchy problem \eqref{eq:PE3DBasic}-\eqref{eq:basicInitialCond} given above models 
regional oceanic flows.  
We note however that equations of  a quite similar structure may be given that describe 
the atmosphere and the coupled oceanic atmospheric system.  
See e.g. \cite{PetcuTemamZiane}. The methods developed
here could thus be extended to treat these systems.

The manuscript is organized as follows.  In the initial
Section~\ref{sec:MathNot} we set the mathematical
background for the work defining precisely the notion of pathwise
solutions we are studying.  
Section~\ref{sec:GlobExistenceCriteria} recalls the results 
in \cite{DebusscheGlattHoltzTemam1}
that guarantee the local existence of solutions.  Crucially,
these results imply a maximal time of existence $\xi = \xi(\omega)$;
on those samples $\omega$ (in the probability space $\Omega$)
where $\xi(\omega) < \infty$
certain norms of the solution $U = (\mathbf{v}, T,S)$, 
in particular the $H^1$ norm, must blow up at $\xi(\omega)$.
Having exhibited these preliminaries we next introduce a 
criteria for global existence based on the uniform control in
time of $\mathbf{v} \in L^4$ and $\pd{z} U \in L^{2}$.
This sets the agenda for the remainder of the paper.
In Section~\ref{sec:L4Est} we carry out the estimates in $L^4$.  To this
end we introduce a new `shifted' variable $\hat{\mathbf{v}}$ that satisfies a random
PDE that we may analyze pathwise.  In this way we are able to handle
the pressure via the results in \cite{SohrVonWahl1}.  On the other
hand a number of new terms appear in the nonlinear portion of the 
equations for $\hat{\mathbf{v}}$ that we must tackle.  In Section~\ref{sec:VerticleGradEst}
we turn to the estimates for $\pd{z}U$\ in $L^2$.  In this case the 
pressure disappears when we exhibit the evolution for $|\pd{z} U|^2_{L^2}$.
As such we carry out the estimates in the original variable using stochastic
methods: It\={o} calculus, the Burkholder-Davis-Gundy inequality, etc.  
Finally we include in an appendix further details for various technical complements 
used in the body of the work:  the pressure estimates 
for the Stokes equations after  \cite{SohrVonWahl1} and a particular
version of the Gronwall lemma that we use to close the $L^4$ estimates
in Section~\ref{sec:L4Est}.

\section{Mathematical Background and Notational Conventions}
\label{sec:MathNot}

In order to introduce a precise mathematical definition
of solutions for the stochastic Primitive Equations we begin by rewriting 
\eqref{eq:PE3DBasic} in a slightly different form.
This formulation will be the basis for all that follows below:
\begin{subequations}\label{eq:PE3DInt}
  \begin{gather}
    \begin{split}
    \pd{t} \mathbf{v} 
    &+ (\mathbf{v} \cdot \nabla)\mathbf{v} + w(\mathbf{v}) \pd{z}\mathbf{v}
    + \frac{1}{\rho_0} \nabla p_s 
    - g \int_z^0 \left( \beta_T  \nabla T  + \beta_S  \nabla S  \right) d \bar{z}\\
    &+ f \mathbf{k} \times \mathbf{v} 
    - \mu_{\mathbf{v}} \Delta \mathbf{v} 
    - \nu_{\mathbf{v}} \pd{zz} \mathbf{v}
    = F_{\mathbf{v}} + \sigma_{\mathbf{v}}(\mathbf{v},T,S) \dot{W}_1,\\
    \end{split}
    \label{eq:MomentumPEint}\\
    w(\mathbf{v}) = \int^0_z \nabla \cdot \mathbf{v} d\bar{z},  \quad
    \int_{-h}^0 \nabla \cdot \mathbf{v} d\bar{z} = 0, 
    \label{eq:divFreeTypeCondPEInt}\\
    \pd{t} T + (\mathbf{v}\cdot \nabla) T
            + w(\mathbf{v}) \pd{z} T
            - \mu_{T} \Delta T
            - \nu_{T} \pd{zz} T
            = F_{T} + \sigma_{T}(\mathbf{v},T,S) \dot{W}_2,
    \label{eq:diffEqnTempPEInt}\\
    \pd{t} S + (\mathbf{v}\cdot \nabla) S
             + w(\mathbf{v}) \pd{z} S
             - \mu_{S} \Delta S
             - \nu_{S} \pd{zz} S
             = F_{S} + \sigma_{S}(\mathbf{v},T,S) \dot{W}_3.
    \label{eq:diffEqnSaltPEInt}
 \end{gather}
\end{subequations}
Here, as above in \eqref{eq:PE3DBasic}, $U := (\mathbf{v},T,S)= (u,v, T, S)$, are the
horizontal velocity, temperature and salinity of the fluid under
consideration, $\mu_{\mathbf{v}}$, $\nu_{\mathbf{v}}$, $\mu_{T}$, $\nu_{T}$,
$\mu_{S}$, $\nu_{S}$ are the coefficients of the eddy and molecular
viscosity and the heat and saline diffusivity respectively, $f$ is the Coriolis
(rotation) parameter.
By integrating \eqref{eq:HydroStaticPE}
and making use of the relation \eqref{eq:linearDensityDependence}
we find that the pressure $p$ may be decomposed into a `surface pressure', 
$p_{s}$ and some lower order terms that couple the momentum equations
to those for the temperature and salinity.  Crucially, we note that $p_s$  
does not depend on the vertical variable $z$.
Of course, this system \eqref{eq:PE3DInt} is
supplemented with initial and boundary conditions 
as given in \eqref{eq:basicInitialCond} and 
\eqref{eq:3dPEPhysicalBoundaryCondTop}--\eqref{eq:3dPEPhysicalBoundaryCondSide} above.

\subsection{Abstract Setting for the Equations}

We next recall the abstract setting for \eqref{eq:PE3DInt}
(equivalently \eqref{eq:PE3DBasic}). 
Note that our presentation and notations closely follow the 
recent survey \cite{PetcuTemamZiane}.

Let us first recall some Hilbert spaces associated to \eqref{eq:PE3DInt}.  
Define
\begin{displaymath}
  \begin{split}
    H := \biggl\{ (\mathbf{v}, T,S) &\in (L^{2}(\mathcal{M}))^{4} :
	 \nabla \cdot \int_{-h}^{0} \mathbf{v} dz = 0 \textrm{ in } \mathcal{M}_{0},
	 \mathbf{n}_H \cdot  \int_{-h}^{0} \mathbf{v} dz = 0 \textrm{ on } \partial \mathcal{M}_{0}, 
	 \int_{\mathcal{M}} T \dM =  \int_{\mathcal{M}} S \dM  = 0 \biggr\}.
   \end{split}
\end{displaymath}
We equip this space with the classical $L^{2}$ inner product\footnote{One 
  sometimes also finds the more general definition
  $(U, U^{\sharp}) := \int_{\mathcal{M}} (\mathbf{v}\cdot
  \mathbf{v}^\sharp d + \kappa_{T}  T T^\sharp   + \kappa_{S} S S^{\sharp} )d
  \mathcal{M}$ with $\kappa_{T}, \kappa_{S} > 0$ fixed constants. These parameters $\kappa_{T}, \kappa_{S}$
  are useful for the coherence of physical dimensions and for (mathematical)
  coercivity. Since this is not needed here we take $\kappa_{T} = \kappa_{S} =1$.
  Similar remarks also apply to the space $V$.} which we denote by $| \cdot |$.
Define $P_{H}$ to be the Leray type projection operator from $L^{2}(\mathcal{M})^{4}$ onto $H$.   For $H^{1}(\mathcal{M})^{4}$
we consider the subspace:
\begin{displaymath}
  \begin{split}
    V := \biggl\{ (\mathbf{v}, T,S) \in (H^{1}(\mathcal{M}))^{4} :&
	 \nabla \cdot \int_{-h}^{0} \mathbf{v} dz = 0 \textrm{ in } \mathcal{M}_{0},
	 \mathbf{v} = 0  \textrm{ on } \Gamma_{l},
	 \int_{\mathcal{M}} T \dM = \int_{\mathcal{M}} S \dM = 0 \biggr\}.
  \end{split}
\end{displaymath}
We equip $V$ with the inner product 
\begin{displaymath}
  \begin{split}
    ((U, U^{\sharp})) :=& ((\mathbf{v},\mathbf{v}^{\sharp}))_{1} 
          +((T,T^{\sharp}))_{2} 
          +((S,S^{\sharp}))_{3},\\
    ((\mathbf{v},\mathbf{v}))_{1} :=&
    \int_{\mathcal{M}} \left( \mu_{\mathbf{v}}
     \nabla\mathbf{v} \cdot \nabla\mathbf{v}^{\sharp}+ 
     \nu_{\mathbf{v}}
     \pd{z}\mathbf{v} \cdot \pd{z}\mathbf{v}^{\sharp}
     \right) \dM,\\
        ((T,T^{\sharp}))_{2} :=&
   \int_{\mathcal{M}} \left(\mu_{T}
    \nabla T \cdot \nabla T^{\sharp}+ \nu_{T}
    \pd{z} T \cdot \pd{z} T^{\sharp}
    \right) \dM,\\
        ((S,S^{\sharp}))_{3} :=&
   \int_{\mathcal{M}} \left(\mu_{S}
    \nabla S \cdot \nabla S^{\sharp}+ \nu_{S}
    \pd{z} S \cdot \pd{z} S^{\sharp}
    \right) \dM,\\
\end{split}
\end{displaymath}
and take $\| \cdot \| = \sqrt{((\cdot , \cdot ))}$.
Note that under these definitions a Poincar\'{e} type 
inequality  $|U|\leq c\|U\|$ holds for all $U \in V$.  
We take $V_{(2)}$ to be the closure of $V \cap C^{\infty}(\overline{\mathcal{M}})^{4}$ 
in $(H^2(\mathcal{M}))^{4}$
and equip this space with the classical $H^2(\mathcal{M})$ norm and inner
product.

In the course of the analysis below we shall make estimates involving
the individual components of the solution $U = (u, v, T, S)$.  As such
we shall sometime abuse notation and use $| \cdot |$ and $\| \cdot \|$ 
in the obvious way for $\mathbf{v}=(u,v)$, $T$  or $S$.
We shall also work with the $L^p = L^p(\mathcal{M})$ norms of
these individual components of the solution.  For $p \geq 1$ we denote
$\mathbf{v}^p = (u|u|^{p-1}, v|v|^{p-1})$ and let
$    | \mathbf{v} |_{L^p} :=  
    	\left( \int_{\mathcal{M}} (|u|^p + |v|^p) \dM \right)^{1/p}.$
Furthermore, for $q, p \geq 1$, we write
\begin{displaymath}
  |\mathbf{v}|_{L^q_{\mathbf{x}}L^p_z}
  := 	\left( \int_{\mathcal{M}_0} \left(
    \int_{-h}^0 (|u|^p + |v|^p) dz
    \right)^{q/p}
 \dMo \right)^{1/q}.
\end{displaymath}
\begin{Rmk}\label{rmk:embeddingLpLq}
  Assume that $\mathbf{v} \in H^1(\mathcal{M})$, with $\mathbf{v} =0$ 
  on $\Gamma_{l}$ or $\int_{\mathcal{M}} \mathbf{v} \dM = 0$.  
  As in \cite{PetcuTemamZiane},  an elementary calculation that makes 
  use of the Sobolev embedding theorem in $\mathbb{R}^{2}$, reveals that
  \begin{equation}  \label{eq:sobEmbAniso}
    |\mathbf{v}|_{L^q_\mathbf{x} L^2_z } \leq c
        |\mathbf{v}|^{1-s} \|\mathbf{v}\|^{s},
  \end{equation}
  where $q \geq 2$ and $s = 1 - 2/q$.  This observation
  will be used on several occasions below.
\end{Rmk}

The principal linear portion of the equation is defined by\footnote{In comparison
to previous works, such as  \cite{PetcuTemamZiane},
we do not include all of the terms due to the pressure in the 
definition of $A$.  Such elements destroy the symmetry of $A$ and are 
therefore relegated to a lower order term $A_{p}$.  See 
\eqref{eq:linLowerOrderDef}.}
\begin{displaymath}
  AU = P_H \left(
  \begin{split}
    -\mu_{\mathbf{v}} \Delta \mathbf{v} - \nu_{\mathbf{v}} \pd{zz} \mathbf{v} \\
    -\mu_{T} \Delta T - \nu_{T} \pd{zz} T \\
    -\mu_{S} \Delta S - \nu_{S} \pd{zz} S \\
 \end{split}
  \right), \quad
  \textrm{ for any } U = (\mathbf{v}, T, S) \in D(A) 
\end{displaymath}
where:
\begin{equation}\label{eq:DAdef}
  \begin{split}
    D(A) = \{ U = (\mathbf{v}, T) \in V_{(2)}: &  
       \pd{z} \mathbf{v} = 
       \pd{z} T = 
       \pd{z} S = 0 \textrm{ on } \Gamma_i,\\
   & \pd{\mathbf{n}_H} T =  \pd{\mathbf{n}_H} S = 0 \textrm{ on } \Gamma_l, 
    \pd{z}\mathbf{v} = \pd{z} T = \pd{z} S = 0 \textrm{ on } \Gamma_b  
       \}.
  \end{split}
\end{equation}
We observe that $A$ is self adjoint, positive definite.  Indeed, by
integration by parts and using the boundary conditions imposed
by \eqref{eq:DAdef}, we see that
$\langle A U, U^{\sharp} \rangle = ((U,U^{\sharp}))$ for all
$U, U^{\sharp} \in D(A)$ and thus,
by density, for all $U, U^{\sharp} \in V$. Note also that, due to 
regularity results for the Stokes problem of Geophysical Fluid Dynamics,
$|AU| \cong  |U|_{H^2}$.  See \cite{Ziane1} and also \cite{PetcuTemamZiane}.

We next turn to the quadratically nonlinear terms appearing in \eqref{eq:PE3DInt}.
Noting that there is no momentum equation for $w$ in
\eqref{eq:PE3DInt} and in accordance with \eqref{eq:divFreeTypeCondPEInt}
we \emph{define} the diagnostic function:
\begin{equation}\label{eq:diagnosticVal}
  w(U) = w(\mathbf{v}) := \int^0_z \nabla \cdot \mathbf{v} d \bar{z}, \quad
  U = (\mathbf{v}, T, S) \in V.
\end{equation}
Take, for $U, U^{\sharp} \in D(A)$:
\begin{equation}\label{eq:NLTerm1}
  B_1(U,U^\sharp) :=  P_H \left( 
  \begin{split}
    (\mathbf{v} \cdot \nabla)\mathbf{v}^\sharp\\
   (\mathbf{v} \cdot \nabla) T^\sharp\\
   (\mathbf{v} \cdot \nabla) S^\sharp\\
  \end{split}
  \right), \quad
  B_2(U,U^\sharp) :=  P_H \left( 
  \begin{split}
    w(\mathbf{v}) \pd{z} \mathbf{v}^\sharp\\
    w(\mathbf{v}) \pd{z} T^\sharp \\
    w(\mathbf{v}) \pd{z} S^\sharp \\
  \end{split}
  \right).
 \end{equation}
We let $B(U, U^\sharp) := B_{1}(U, U^\sharp) + B_{2}(U, U^\sharp)$,
and often write $B(U) = B(U,U)$.
As in \cite{PetcuTemamZiane} one may show that $B$
is well defined as an element in $H$ for any $U, U^{\sharp} \in D(A)$ or $V_{(2)}$.
In addition to the properties of $B$  appearing in \cite{PetcuTemamZiane}
we have the following additional bounds, which are established with
anisotropic estimates along the same lines (cf. Remark~\ref{rmk:embeddingLpLq}).
\begin{Lem}\label{eq:Best}
Suppose that $U, U^\sharp \in D(A)$ and that $U^\flat \in H$.  Then 
\begin{equation}\label{eq:BestL6VertGrad}
\begin{split}
   |\langle B(U,U^\sharp), U^\flat \rangle|
   	\leq& c (
	|\mathbf{v}|_{L^4} \|U^\sharp\|^{1/4} |A U^\sharp|^{3/4} |U^\flat|
	 +
	\| \mathbf{v}\|^{1/2} | \mathbf{v}|^{1/2}_{(2)}
	|\pd{z} U^\sharp|^{1/2} \|\pd{z} U^\sharp\|^{1/2} |U^\flat|).
\end{split}
\end{equation}
\end{Lem}
For the second component of the pressure in \eqref{eq:MomentumPEint} we take
\begin{equation}\label{eq:linLowerOrderDef}
  A_p U =   P_H \left(
  \begin{array}{c}
    - g \int_z^0 \left( \beta_T  \nabla T  + \beta_S  \nabla S  \right) d \bar{z} \\
    0\\
    0
  \end{array}
  \right), \quad U \in V.
\end{equation}
We capture the Coriolis (rotational) forcing according
to
\begin{equation}
  \label{eq:CorTerm}
  E U = P_H \left(
  \begin{array}{c}
    f \mathbf{k} \times v \\
    0\\
    0
  \end{array}
  \right), \quad U \in H.
\end{equation}
Finally we set
\begin{equation}\label{eq:DetForcingTerm}
    F = P_H \left(
  \begin{array}{c}
    F_{\mathbf{v}}\\
    F_{T}\\
    F_{S}\\
  \end{array}
  \right).
\end{equation}
We shall assume throughout this work that: 
\begin{equation}\label{eq:SizeConF}
   F \in  L^2(\Omega; L^2_{loc}([0, \infty); L^4(\mathcal{M}))).
\end{equation}

We finally give a precise definition for the stochastic terms appearing 
in \eqref{eq:PE3DInt} (i.e. \eqref{eq:PE3DBasic}).  For this
purpose let us briefly recall some aspects of the theory of
the infinite dimensional It\={o} integration. 
As we are studying \emph{pathwise solutions} of \eqref{eq:PE3DBasic}
(see Definitions~\ref{def:SolnDefloc},~\ref{def:SolnDefMax} below) 
we shall fix throughout this work a single stochastic
basis $\mathcal{S} := (\Omega, \mathcal{F}, \Prb, 
\{\mathcal{F}_t\}_{t \geq 0}, W)$.  Here $W$ is a cylindrical 
brownian motion defined on an auxiliary Hilbert space
$\mathfrak{U}$ and adapted to the filtration $\{\mathcal{F}_t\}_{t \geq 0}$.  
By picking a complete orthonormal basis $\{e_k\}_{k \geq 1}$ 
for $\mathfrak{U}$, $W$ may be written as the formal sum
$W(t,\omega) = \sum_{k \geq 1} e_k W_k(t, \omega)$ where
the elements $W_k$ are an independent sequence of 1D standard
Brownian motions.

Consider another separable Hilbert space $X$ and let 
$L_2(\mathfrak{U},X) = \{ R \in \mathcal{L}(\mathfrak{U}, X) : \sum_k |Re_k|^2 < \infty\}$,
that is the collection of Hilbert-Schmidt operators from $\mathfrak{U}$ into 
$X$.
Given an $X$-valued predictable\footnote{Let 
$\Phi = \Omega \times [0,\infty)$ and take
$\mathcal{G}$ to be the $\sigma$-algebra generated by sets of the
form
\begin{displaymath}
    (s,t] \times F, \quad 0 \leq s< t< \infty, F \in \mathcal{F}_s;
    \quad \quad
    \{0\} \times F, \quad F \in \mathcal{F}_0.
\end{displaymath}
Recall that an $X$ valued process $U$ is called predictable (with
respect to the stochastic basis $\mathcal{S}$) if it is measurable
from $(\Phi,\mathcal{G})$ into $(X, \mathcal{B}(X))$,
$\mathcal{B}(X)$ being the family of Borel sets of $X$.}
process
$G \in L^{2}(\Omega; L^{2}_{loc}$ 
$([0, \infty),L_{2}(\mathfrak{U}, X)))$ 
one may define the 
(It\={o}) stochastic integral
\begin{displaymath}
   M_{t} := \int_{0}^{t} G dW = \sum_k \int_0^t G_k dW_k,
\end{displaymath}
as an element in $\mathcal{M}^2_X$, that is the space of all
$X$-valued square integrable martingales (see 
\cite{ZabczykDaPrato1} or \cite{PrevotRockner}); here $G_k = G e_k$.  
The process $\{M_t \}_{t \geq 0}$ has
many desirable properties.  Most notably for the analysis here,
the Burkholder-Davis-Gundy inequality holds which in the present
context takes the form,
\begin{equation}\label{eq:BDG}
  \E \left(\sup_{t \in [0,T]} \left| \int_0^t G dW  \right|_X^r \right)
  \leq c \E \left(
    \int_0^T |G|_{L_2(\mathfrak{U}, X)}^2 dt \right)^{r/2},
\end{equation}
valid for any $r \geq  1$.  Here $c$ is an absolute constant depending
only on $r$.

Given any pair of Banach spaces $\mathcal{X}$ and $\mathcal{Y}$ we denote by
$Bnd_u(\mathcal{X}, \mathcal{Y})$, the collection of all continuous mappings 
$\Psi: [0, \infty) \times \mathcal{X} \rightarrow \mathcal{Y}$
such that
\begin{displaymath}
    \| \Psi(x,t) \|_{\mathcal{Y}} \leq c(1 + \|x\|_{\mathcal{X}}), \quad x
    \in \mathcal{X}, t \geq 0,\\
\end{displaymath}
where the numerical constant $c$ may be chosen
independently of $t$.  If, in addition,
\begin{displaymath}
   \| \Psi(x,t) - \Psi(y, t)\|_{\mathcal{Y}} \leq c \|x - y \|_{\mathcal{X}},
   \quad x, y \in \mathcal{X}, t \geq 0,\\
\end{displaymath}
we say $\Psi$ is in $Lip_u(\mathcal{X}, \mathcal{Y})$.
We define
\begin{equation}\label{eq:nonLinearNoisePE}
    \sigma((\mathbf{v}, T,S)) = \sigma(U) = P_{H}
    \left(
      \begin{array}{c}
    \sigma_{\mathbf{v}}(U)\\
    \sigma_{T}(U)\\
    \sigma_{S}(U)\\
  \end{array}
      \right),  \quad U \in H,
\end{equation}
and assume that
$\sigma: [0, \infty) \times H \rightarrow L_2(\mathfrak{U},H)$
with
\begin{equation}\label{eq:lipCondSig}
  \sigma \in Lip_u(H, L_2(\mathfrak{U}, H)) 
   \cap Lip_u(V, L_2(\mathfrak{U}, V)) 
   \cap Bnd_u(V, L_2(\mathfrak{U}, D(A))).
\end{equation}
Under the assumption \eqref{eq:lipCondSig} on $\sigma$, the
stochastic integral $t \mapsto \int_0^t \sigma(U) dW$ may be shown to
be well defined (in the It\={o} sense), taking values in $H$ whenever
$U \in L^{2}(\Omega, L^{2}_{loc}([0,\infty); H))$ and is predictable.

\begin{Rmk}\label{rmk:PossibleNoiseStructures}
The condition \eqref{eq:lipCondSig} may be shown to cover 
a wide class  of examples, including but not
limited to the classical cases of additive and linear 
multiplicative noise, projections of the solution in any
direction, and directional forcings of Lipschitz functionals 
of the solution.  See \cite{GlattHoltzTemamTribbia1} for further
details and physical motivations.

We note that the final condition 
$\sigma \in Bnd_u(V, L_2(\mathfrak{U}, D(A)))$ in 
fact may be weakened slightly to  
$\sigma \in Bnd_u(D(A),$ $L_2(\mathfrak{U}, D(A)))$ for the proof 
of local existence.  However for the global existence
we need this stronger condition in
order to carry out the $L^{4}$ estimates which appear in
Section~\ref{sec:L4Est}.  In fact 
the stronger condition used in this section
is an artifact of a `pathwise' approach
which further complicates the nonlinear
term. This pathwise approach is necessitated by
the appearance of certain terms explicitly 
involving the pressure $p_s$.
\end{Rmk}

Collecting the operators defined above we reformulate
\eqref{eq:PE3DInt} (equivalently, \eqref{eq:PE3DBasic}) as 
the following abstract evolution system
\begin{equation}\label{eq:AbstractFormulationPE}
  dU + (AU +  B(U) +A_p U +  E U)dt = Fdt + \sigma(U) dW,
  \quad U(0) = U_0.
\end{equation}
We recall the following pathwise (or probabilistically strong) notions of local and
global existence for this system.
\begin{Def}[Local Solutions]\label{def:SolnDefloc}
  Let $\mathcal{S} = (\Omega, \mathcal{F}, \{\mathcal{F}_t\}_{t \geq
    0}, \mathbb{P}, W)$ be a fixed stochastic basis and suppose that
  $U_0$ is a $V$ valued, $\mathcal{F}_0$ measurable random variable
  with $\E \|U_0\|^2 < \infty$.  
  Assume that $F$ satisfies  \eqref{eq:SizeConF} and that
  \eqref{eq:lipCondSig} holds for $\sigma$.
  \begin{itemize}
  \item[(i)] A pair $(U, \tau)$ is a \emph{a local strong pathwise solution} of
   \eqref{eq:AbstractFormulationPE} (i.e. of \eqref{eq:PE3DBasic})
   if $\tau$ is a strictly positive stopping
    time and $U(\cdot \wedge \tau)$ is an $\mathcal{F}_{t}$-adapted
    process in $V$ so that (relative to the fixed basis $\mathcal{S}$)
    \begin{equation}\label{eq:regularityCondMG}
      \begin{split}
        U(\cdot \wedge \tau) \in  L^2(\Omega; C ([0,\infty); V)),\quad 
        U \indFn{t \leq \tau} \in  L^2(\Omega;
        L^2_{loc}([0,\infty); D(A)));
      \end{split}
  \end{equation}
    and, for every $t \geq 0$,
    \begin{equation}\label{eq:spdeAbstracMG}
      \begin{split}
        U(t \wedge \tau)
           &+ \int_0^{t \wedge \tau} 
           (A U +  B(U) + A_p U + E U)
           dt'
           = U_{0}
           + \int_0^{t \wedge \tau} F dt' 
            + \int_0^{t \wedge \tau} \sigma(U) dW,
     \end{split}
  \end{equation}
  with equality understood in $H$.
  \item[(ii)] Strong pathwise solutions of   \eqref{eq:AbstractFormulationPE} are said to be
   \emph{(pathwise) unique} up to a stopping time $\tau > 0$ if given
   any pair of solutions $(U^1,\tau)$, $(U^2,\tau)$ which
   coincide at $t = 0$ on the event
   $\tilde{\Omega} = \{U^1(0) = U^2(0)\} \subseteq \Omega$, then
   \begin{displaymath}
    \Prb \left( \indFn{\tilde{\Omega}} ( U^1(t \wedge \tau) - U^2(t \wedge
      \tau))  = 0; \forall t \geq 0 \right) = 1.
   \end{displaymath}
\end{itemize}   
\end{Def}
\begin{Rmk}\label{rmk:WankingBSAboutBasicDef}
  For a local solution $(U,\tau)$ consider the increasing function
  $$
  	\phi(\sigma) := \sup_{t \in [0, \sigma]} \|U(t)\|^{2}
	\quad \quad (\textrm{defined for } \sigma \in [0, \tau]).
  $$
  We infer from \eqref{eq:regularityCondMG} that $\phi(\tau) < \infty$,
  almost surely.  Of course the question arrises as to whether the solution $U$
  and therefore $\phi$ can be extended beyond $\tau$.  We will call $\xi$ the
  supremum of the $\tau$'s such that $(U,\tau)$ is a local solution of \eqref{eq:AbstractFormulationPE}.  A priori,
  $\xi$ might be finite or infinite and we are lead to the following notions of
  maximal and global solutions.
\end{Rmk}
\begin{Def}[Maximal and Global Solutions]\label{def:SolnDefMax}
\mbox{}
\begin{itemize}  
\item[(i)] Suppose that $\{\tau_n\}_{n\geq 1}$ is a an increasing sequence
    of stopping times converging to a (possibly infinite) stopping time
    $\xi$ and assume that $U$ is a predictable process in $H$.  We
   say that the pair $(U,\xi)$ is \emph{a maximal pathwise
     strong solution} if $(U, \tau_n)$ is a local strong pathwise solution for
   each $n$ and
  \begin{equation}\label{eq:FiniteTimeBlowUp}
     \phi(\xi) = \sup_{t \in [0, \xi]} \|U\|^2 
     = \infty
  \end{equation}
  almost surely on the set $\{\xi < \infty\}$. 
  \item[(ii)] If $(U, \xi)$ is a maximal pathwise strong solution and $\xi =
    \infty$ a.s. then we say that the solution is global.  Note that $\phi(\infty)$
    can be finite or infinite in this case.
  \end{itemize}
\end{Def}

With these foundations in place we finally state, in precise 
terms, the main result of the work, the global well-posedness
of \eqref{eq:PE3DBasic} in the class of pathwise, strong solutions:
\begin{Thm}\label{thm:MainResult}
  Suppose that $U_0 \in L^{2}(\Omega,V)$ and is $\mathcal{F}_0$
  measurable.  Assume that $F$ satisfies  \eqref{eq:SizeConF} and that
  \eqref{eq:lipCondSig} holds for $\sigma$. Then, there exists 
  a unique global pathwise solution $U$ of
  \eqref{eq:AbstractFormulationPE} with $U(0) = U_0$.
\end{Thm}
The proof of this Theorem is given below at the end of 
Section~\ref{sec:GlobExistenceCriteria}.  It follows 
directly from the combination of Theorems~\ref{thm:LocalExistence},
\ref{thm:StrongNormBnds} and Propositions~\ref{thm:L4MaximalExistence},
\ref{thm:VerticalGradConclusion}.


\section{Local Existence and A Criteria For Global Existence}
\label{sec:GlobExistenceCriteria}

As with the deterministic case, the first step
in the analysis is to establish the local existence
of solutions all the way to a `maximal' existence
time $\xi > 0$.  Recall that if $\xi(\omega)< \infty$ for $\omega \in \Omega$,
then the $H^{1}$ norm of the solution must blow up 
at this maximal time as expressed by \eqref{eq:FiniteTimeBlowUp}. 
With this in hand, we next show in 
Theorem~\ref{thm:StrongNormBnds}
that the strong norms of solutions may be controlled by a
combination of  the $L^4$ norm of the momentum $\mathbf{v}$
and by the vertical gradient of the entire solution $U$.   
The results in Theorem~\ref{thm:StrongNormBnds} establish
a sufficient condition for Theorem~\ref{thm:MainResult} to be valid 
and thus set the program for the rest of the article.

The first step, establishing the existence of a maximal pathwise
solution, is given by Theorem~\ref{thm:LocalExistence} immediately
below.  The proof of this result follows immediately from abstract results 
in \cite[Theorem 2.1 and Section 6]{DebusscheGlattHoltzTemam1}.
\begin{Thm}\label{thm:LocalExistence}
  Suppose that \eqref{eq:SizeConF} and \eqref{eq:lipCondSig} hold
  for $F$ and $\sigma$ respectively. Given any $\mathcal{F}_0$-measurable $U_0 \in L^2(\Omega, V)$, there exists a 
  unique maximal pathwise solution $(U, \xi)$ of
  \eqref{eq:AbstractFormulationPE} with $U(0) = U_0$.
  If, for some $p \geq 2$,
  $U(0) \in L^p(\Omega, V)$ then, for any $t > 0$,
  we have the bounds
  \begin{equation}\label{eq:LpInTimeWeakBnds1}
    \E \left(  \sup_{t' \in [0, t \wedge \xi]} |U|^p
      + \int_0^{t \wedge \xi} \|U\|^2|U|^{p -2} dt'
    \right) < \infty,
  \end{equation}
  and
  \begin{equation}\label{eq:LpInTimeWeakBnds2}
    \E \left( \int_0^{t \wedge \xi} \|U\|^2 dt'
      \right)^{p/2} < \infty.
  \end{equation}
\end{Thm}  

\begin{Rmk}\label{rmk:OtherMethods}
  The local existence of solutions for a (nonlinear) multiplicative noise
  is more challenging than the additive case.  In this latter situation where
  the noise terms have no state dependence, a simple
  change of variables which `subtracts off' the noise
  can be used to transform the governing equations
  into a random PDE which can be treated using the
  methods of ordinary calculus.  Here, crucially, the probabilistic dependence, given as 
  $\omega \in \Omega$, may be essentially treated as a parameter 
  in the problem.  
  
  In the general case of nonlinear multiplicative noise
  no such transformation is available.  As a consequence
   stochastic integral terms must be estimated directly and as such
  the classical compactness methods break down since 
  the problem may no longer be handled with the `pathwise' approach 
  typically used in the additive case.
  More involved probabilistic machinery is therefore needed.
  The proof of Theorem~\ref{thm:LocalExistence} appearing in
  \cite{DebusscheGlattHoltzTemam1} is based on an elementary but
  powerful characterization of convergence in probability given in 
  \cite{GyongyKrylov1}.  With this characterization we were able to
  establish a Yamada-Watanabe type result namely: 
  `pathwise' solutions follow from the existence of Martingale solutions
  when we have pathwise uniqueness for solutions 
  (see Definition~\ref{def:SolnDef},(ii)). A different method based 
  on the Cauchy convergence of the Galerkin solutions associated
  to the basic SPDE is developed in \cite{GlattHoltzZiane2} for
  the Navier-Stokes equations in dimensions $D=2,3$ and for the $2D$
  Primitive equations in \cite{GlattHoltzTemam1, GlattHoltzTemam2}.
\end{Rmk}

For any give $K > 0$ we define the following stopping times which are 
used here and below:
  \begin{align}\label{eq:WeakExistence}
    \tau_K^W :=& \inf_{t \geq 0} \left\{
      \sup_{t' \in [0, t \wedge \xi]} |U|^2+
       \int_0^{t \wedge \xi} (\|U\|^2 + |F|^2_{L^4}) dt' \geq K
    \right\} \notag\\
    =& \sup_{t \geq 0} \left\{
      \sup_{t' \in [0, t \wedge \xi]} |U|^2+
       \int_0^{t \wedge \xi} (\|U\|^2 + |F|^2_{L^4}) dt' \leq K
    \right\},
  \end{align}
and also
\begin{align}
   \tau_K^{(1)} := \inf_{t \geq 0} \left\{
   	\sup_{t' \in [0,t \wedge \xi]} | \mathbf{v} |_{L^4}^4
		\geq K \right\}
		=& \sup_{t \geq 0} \left\{
   	\sup_{t' \in [0,t \wedge \xi]} | \mathbf{v} |_{L^4}^4
		\leq K \right\}, \notag\\
 \tau_K^{(2)} := \inf_{t \geq 0} \left\{
	\sup_{t' \in [0,t \wedge \xi]} | \pd{z} U |^2 
	+ \int_0^{t \wedge \xi} \| \pd{z} U \|^2
		\geq K
	\right\}
	=& \sup_{t \geq 0} \left\{
	\sup_{t' \in [0,t \wedge \xi]} | \pd{z} U |^2 
	+ \int_0^{t \wedge \xi} \| \pd{z} U \|^2
		\leq K
	\right\}, \notag\\
	\tau_{K} :=& \tau_{K}^{(1)} \wedge \tau_{K}^{(2)}.
	\label{eq:BlowUpControllerTm}
\end{align}
Note here that we employ notation typically used 
in optimization (see e.g. \cite{EkelandTemam1999, Rockafellar1970, Polak1971,Polak1997})
and we follow the convention that the infimum of
an empty set is $+\infty$.
By Theorem~\ref{thm:LocalExistence} and
the standing assumption \eqref{eq:SizeConF}, it is direct to  
infer that $\lim_{K \uparrow \infty} \tau_K^W = \infty$.
We establish below, by gathering together the 
conclusions in Propositions~\ref{thm:L4MaximalExistence},
\ref{thm:VerticalGradConclusion},
that $\lim_{K \uparrow \infty}\tau_K = \infty$.  This, as the 
next result shows, is a sufficient condition for the global 
existence of solutions of \eqref{eq:PE3DBasic}.   

We observe and emphasize
that the definition of $\tau_{K}$ \emph{does not} imply that
$\tau_{K} \leq \xi$ a.s.
Indeed, it could be possible that $\tau_K=\infty$ if 
the sets defining $\tau_K$ are empty.
  Rather, that $\tau_{K} \leq \xi$, will be proven as part of the 
next theorem.

\begin{Thm}\label{thm:StrongNormBnds}
The hypotheses are the same as in Theorem~\ref{thm:LocalExistence}
and we consider the maximal strong pathwise solution  $(U, \xi) = ((\mathbf{v}, T, S), \xi)$
of \eqref{eq:PE3DBasic} define in this theorem.
For any deterministic constant $K$, consider the stopping 
times $\tau_K$ defined according to \eqref{eq:BlowUpControllerTm}.  
Then, for any  deterministic time $t > 0$ and any $K >0$,
\begin{equation}\label{eq:ExistenceUptoControlTimes}
\E \left( 
 	\sup_{t' \in [0, t \wedge \tau_K \wedge \xi]} \|U \|^2 +
	\int_0^{t \wedge \tau_K \wedge \xi} |AU|^2 dt'
     \right) < \infty.
\end{equation}
Furthermore $\tau_{K} \leq \xi$ almost surely and consequently
if we show (by other means) that $\lim_{K \uparrow} \tau_{K} = + \infty$
then $\xi = \infty$ a.s., that is they solution $U$ is global
in the sense of Definition~\ref{def:SolnDefMax}, (ii).


\end{Thm}
\begin{proof}
  With an application of the It\={o} formula we deduce the following
  evolution equation for $\|U\|^2$:
 \begin{displaymath}
    \begin{split}
      d \|U\|^2 &+ 2|AU|^2dt
             = -2\langle B(U) + A_p U + E U + F, AU \rangle dt
             + \|\sigma(U)\|^2_{L_2(\mathfrak{U}, V)} dt
               + 2\langle A^{1/2}\sigma(U) , A^{1/2} U \rangle dW.
    \end{split}
  \end{displaymath}
Fix any $0 \leq \tau_a \leq \tau_b \leq \tau_K \wedge \xi \wedge t$.  Integrating in time,
taking a supremum over the interval $[\tau_a, \tau_b]$ and then the expected
value of the resulting expression yields the estimate
  \begin{displaymath}
    \begin{split}
      \mathbb{E} \biggl(
        \sup_{ \tau_a \leq t' \leq \tau_b} &\|U\|^2
                       + 2 \int_{\tau_a}^{\tau_b} |AU|^2 dt'
        \biggr)\\
        \leq&  \E \| U(\tau_a)\|^2 
        + c \E \int_{\tau_a}^{\tau_b} |\langle B(U), AU \rangle| dt
       + c \E \int_{\tau_a}^{\tau_b} ( 
                       |\langle A_p U + E U + F,AU\rangle|
       	         + \|\sigma(U)\|^2_{L_2(\mathfrak{U}, V)} )dt\\
        &+ c \mathbb{E} \sup_{\tau_a \leq t' \leq \tau_b} 
        \left|\int_{\tau_a}^{t'} \langle A^{1/2}\sigma(U) , A^{1/2} U \rangle
                 dW \right|\\
         :=& \E\|U(\tau_a)\|^2 + J_1 +  J_2 + J_3.\\
    \end{split}
  \end{displaymath}
  We address the terms $J_k$ in reverse order.  For $J_3$ using 
  the Burkholder-Davis-Gundy inequality, \eqref{eq:BDG}, and the assumption in
  \eqref{eq:lipCondSig} that $\sigma \in Lip_u(V, L_2(\mathfrak{U}, V))$  we infer
  \begin{displaymath}
  \begin{split}
    J_3 &\leq
    c \E \left(\int_{\tau_a}^{\tau_b} \langle A^{1/2}\sigma(U) , A^{1/2} U \rangle^2 dt \right)^{1/2}
    \leq     c \E \left(\int_{\tau_a}^{\tau_b} \|\sigma(U)\|^{2}_{L_{2}(\mathfrak{U}, V)} \| U \|^2 dt \right)^{1/2}\\
    &\leq
     \frac{1}{2} \E \sup_{t' \in [\tau_a, \tau_b]} \|U \|^2 + 
       c \E \int_{\tau_a}^{\tau_b} (1 + \|U\|^2) dt.
  \end{split}
  \end{displaymath}
  For $J_{2}$ we use again that $\sigma \in Lip_u(V, L_2(\mathfrak{U}, V))$
  and make direct estimates for $A_{p}$ and $E$; see respectively, 
  \eqref{eq:linLowerOrderDef}, \eqref{eq:CorTerm} above.
  We infer
  \begin{displaymath}
     |\langle A_p U + E U + F,AU\rangle|
       	         + \|\sigma(U)\|^2_{L_2(\mathfrak{U}, V)}  \leq \frac{1}{2} |AU|^2  + c(1 + \|U\|^2 + |F|^2).
  \end{displaymath}
  On the other hand, \eqref{eq:BestL6VertGrad} implies that
  \begin{displaymath}
    \begin{split}
    |\langle B(U), AU \rangle|
         &\leq 
       	|\mathbf{v}|_{L^4} \|U\|^{1/4} |A U|^{7/4} +
	\|U\|^{1/2}|\pd{z} U|^{1/2} \|\pd{z} U\|^{1/2}
        |AU|^{3/2})\\
        &\leq \frac{1}{2} |AU|^2+
        c(|\mathbf{v}|_{L^4}^8 +|\pd{z} U|^2 \|\pd{z} U\|^2) \|U\|^2.
   \end{split}
 \end{displaymath}      
 This takes care of $J_{1}$.
 Combining these three estimates, we conclude that
 \begin{equation}\label{eq:FinalGronwallSetUpStrongSolns}
   \begin{split}
   \E & \left(
       \sup_{ \tau_a \leq t' \leq \tau_b} \|U\|^2
                      + \int_{\tau_a}^{\tau_b} |AU|^2 dt'
       \right)\\
       &\quad \quad \quad \leq c \E  \left( \| U(\tau_a)\|^2 +
                   \int_{\tau_a}^{\tau_b} (|\mathbf{v}|_{L^4}^8
                 +|\pd{z} U|^2 \|\pd{z} U\|^2 +1 ) \|U\|^2 dt
               +  \int_{\tau_a}^{\tau_b} |F|^2 dt \right).
  \end{split}
\end{equation}

We now apply a stochastic version of the Gronwall lemma to \eqref{eq:ExistenceUptoControlTimes} 
from \cite[Lemma 5.3]{GlattHoltzZiane2}
which is recalled below in the Appendix as Proposition~\ref{thm:semiMGGronwall}.   
Note carefully that the constants appearing in the above equation do not depend on $\tau_a$,$\tau_b$.
By taking
$X:= \|U\|^2$, $Y:= |AU|^2$, $Z:= |F|^2$, $R:=|\mathbf{v}|_{L^4}^8
                 +|\pd{z} U|^2 \|\pd{z} U\|^2 +1$ and finally $\tau := t \wedge \tau_K \wedge \xi$
in Proposition~\ref{thm:semiMGGronwall} we therefore derive \eqref{eq:ExistenceUptoControlTimes}
from \eqref{eq:FinalGronwallSetUpStrongSolns}.


We now argue by contradiction to show that $\tau_{K} \leq \xi$. 
If, on a set of non-zero measure, $\tau_{K} > \xi$
then using that $\{\tau_K > \xi\} = \cup_{t \geq 0} \{ \tau_K \wedge t > \xi\}$
we would infer the existence of a deterministic time $t >0$ such that
$\{\tau_K \wedge t > \xi\}$ is of non-zero measure.  By the
definition of $\xi$ (cf. \eqref{eq:FiniteTimeBlowUp}), we would therefore infer that
for this $t >0$,
$$
	\sup_{t' \in [0, t \wedge \tau_K \wedge \xi]} \|U \|^2 +
	\int_0^{t \wedge \tau_K \wedge \xi} |AU|^2 dt'  \geq \sup_{t' \in [0,  \xi]} \|U \|^2 = \infty
$$
over $\{\tau_K \wedge t > \xi\}$.  Since $\{\tau_K \wedge t > \xi\}$ is a set of non-zero measure
we have a contradiction with \eqref{eq:ExistenceUptoControlTimes}.
With this, the proof of Theorem~\ref{thm:StrongNormBnds} is therefore complete.
\end{proof}

We conclude this section by explaining in precise terms
how we will use Theorem~\ref{thm:StrongNormBnds} to
prove Theorem~\ref{thm:MainResult}:

\begin{proof}[Proof of Theorem~\ref{thm:MainResult}]
Define the stopping times $\tau_{K}^{(1)}$, $\tau_{K}^{(2)}$,
and take $\tau_{K} := \tau_{K}^{(1)} \wedge \tau_{K}^{(2)}$ as in 
\eqref{eq:BlowUpControllerTm}.  According to 
Theorem~\ref{thm:StrongNormBnds}, the desired result
will follow once it is shown that $\lim_{K \uparrow \infty} 
\tau_{K} = \infty$ a.s.   In Proposition~\ref{thm:L4MaximalExistence}
it is shown that $\lim_{K \uparrow \infty} \tau_{K}^{(1)} = \infty$
and from Proposition~\ref{thm:VerticalGradConclusion} we 
infer that $\lim_{K \uparrow \infty} \tau_{K}^{(2)} = \infty$.  The proof
is therefore complete.
\end{proof}


\section{Estimates in $L^4(\mathcal{M})$}
\label{sec:L4Est}

In this section we establish that $\lim_{K \uparrow \infty} \tau_K^{(1)}  = + \infty$
where $\tau_{K}^{(1)}$ is defined as in \eqref{eq:BlowUpControllerTm}
above.
This amounts to showing that the horizontal component of the
velocity field, $\mathbf{v}$,
of the solution $U = (\mathbf{v}, T, S)$ is almost surely finite in
$L^\infty_tL^4_{\mathbf{x},z}$, at least up to the maximal time of
existence, $\xi$.  

More precisely we will prove the following:
\begin{Prop}\label{thm:L4MaximalExistence}
  Suppose that the conditions given in Theorem~\ref{thm:MainResult}
  hold and consider the resulting maximal strong solution 
  $(U, \xi) = ((\mathbf{v}, T, S), \xi)$ of \eqref{eq:PE3DBasic}. 
  Then, for any $t > 0$,
  \begin{equation}\label{eq:BoundednessUpToExistenceTimeFullEqn}
    \sup_{t' \in [0, t \wedge \xi]} |\mathbf{v}(t')|_{L^4} < \infty
    \quad a.s.
  \end{equation}
Moreover, defining, for $K > 0$ the stopping time $\tau_{K}^{(1)}$ 
as in \eqref{eq:BlowUpControllerTm}
we have, up to a set of measure zero,  that
$\lim_{K \uparrow \infty} \tau_{K}^{(1)} = \infty$.
\end{Prop}
Proposition~\ref{thm:L4MaximalExistence} will be an easy
consequence of Propositions~\ref{thm:ExistenceRegularityAuxLinCheckU},
\ref{thm:ExistenceRegularityAuxLinHatU} below.
The proof relies on a decomposition of the solution $U$ of 
\eqref{eq:spdeAbstracMG}  into a sum of two components:
$\check{U}=(\check{\mathbf{v}}, \check{T}, \check{S})$ and $\hat{U} = (\hat{\mathbf{v}}, \hat{T}, \hat{S})$;  $\check{U}$ solves a linear
stochastic system, \eqref{eq:auxLinSystemAbs}, 
whereas $\hat{U}$ is the solution of the random PDE,
\eqref{eq:UhateqnRandPDE}. 
The later system may be studied `pathwise', that is
pointwise in  $\Omega$.   We are therefore
able to treat certain terms involving the 
pressure explicitly by means of Proposition~\ref{thm:SVWResult}
(cf.  \cite{ZianeKukavica}, \cite{SohrVonWahl1}).  In any case,
we are able to establish control on the norms $L^4_{\mathbf{x},z}$ 
of $\check{\mathbf{v}}$ and $\hat{\mathbf{v}}$ as in \eqref{eq:BoundednessUpToExistenceTimeFullEqn} 
in Propositions~\ref{thm:ExistenceRegularityAuxLinCheckU},
\ref{thm:ExistenceRegularityAuxLinHatU} from which Proposition~\ref{thm:L4MaximalExistence}
then follows immediately.

\subsubsection*{The Decomposition}
\label{sec:aux-lin-sys}

We describe next the decomposition and recall some 
basic properties of the linear portion of this splitting.
Consider the following linear system, an infinite dimensional
Ornstein-Uhlenbeck process:
\begin{equation}\label{eq:auxLinSystemAbs}
  d \check{U} + A \check{U} dt = \indFn{t \leq \xi}  \sigma(U)dW, \quad \check{U}(0) = 0.
\end{equation}
We underline here that the element $(U,\xi)$ appearing in the
right hand side is the unique maximal solution of \eqref{eq:PE3DBasic}
corresponding to the given initial condition $U_0$.
\begin{Prop}\label{thm:ExistenceRegularityAuxLinCheckU}
  There exists a unique global, pathwise strong solution $\check{U}= (\check{\mathbf{v}}, \check{T}, \check{S})$ 
  of \eqref{eq:auxLinSystemAbs} with $\check{U}$ $\in L^{2}(\Omega; C([0, \infty); D(A)))$.
  In particular, for any $t > 0$,
 \begin{equation}\label{eq:BoundednessUpToExistenceTimeChk}
    \sup_{t' \in [0, t \wedge \xi]} |\check{\mathbf{v}}(t')|_{L^4} < \infty
    \quad a.s.
  \end{equation}
\end{Prop}
This result is proven exactly as in \cite[Lemma 4.1]{GlattHoltzTemam2}.  See also
\cite{ZabczykDaPrato1} for further generalities concerning infinite dimensional linear stochastic systems of the 
form \eqref{eq:auxLinSystemAbs}.
\begin{Rmk}\label{rmk:NoiseConditions}
  The final condition in \eqref{eq:lipCondSig}, 
  $\sigma \in Bnd(V, L_2(\mathfrak{U}, D(A)))$ 
  is imposed so that $\check{U}$ is guaranteed to evolve continuously
  in $D(A)$. This stronger conditions is needed in several estimates below
  (see e.g. \eqref{eq:J2EstFinal}, \eqref{eq:J3forL4FinalIntEst}).  
\end{Rmk}

We next define $\hat{U} = U - \check{U}$.  
By subtracting \eqref{eq:auxLinSystemAbs} from \eqref{eq:AbstractFormulationPE}
we find that $\hat{U}$ satisfies
\begin{equation}\label{eq:UhateqnRandPDE}
  \begin{split}
    \frac{d}{dt} \hat{U} + A \hat{U} 
                                + B(\hat{U} + \check{U})
                                +&E(\hat{U} + \check{U}) 
                                + A_p(\hat{U} + \check{U}) = F,\\
                         \hat{U}(0) = U&(0) = U_0.\\
 \end{split}
\end{equation}
The equation for the first component $\hat{\mathbf{v}}$ of the solution 
$\hat{U} = (\hat{\mathbf{v}}, \hat{T}, \hat{S})$
is given by
\begin{equation}\label{eq:evolforShiftedv}
\begin{split}
    \pd{t} \hat{\mathbf{v}}
    + ((\hat{\mathbf{v}} + \check{\mathbf{v}}) \cdot \nabla)
               (\hat{\mathbf{v}} + \check{\mathbf{v}}) 
    &+ w(\hat{\mathbf{v}} + \check{\mathbf{v}}) \pd{z}
         (\hat{\mathbf{v}} + \check{\mathbf{v}})
    + \frac{1}{\rho_0} \nabla \hat{p}_s 
   + f \mathbf{k} \times (\hat{\mathbf{v}} + \check{\mathbf{v}})
    - \mu_{\mathbf{v}} \Delta \hat{\mathbf{v}}
    - \nu_{\mathbf{v}} \pd{zz} \hat{\mathbf{v}}\\
    &= F_{\mathbf{v}}
    +     g \int_z^0 \left( \beta_T  \nabla (\hat{T} + \check{T})
     + \beta_S  \nabla (\hat{S} + \check{S}) \right) d \bar{z}.
\end{split}
\end{equation}
Note that the diagnostic function, $w(\cdot)$ is defined as above in \eqref{eq:diagnosticVal}.
As, with $\mathbf{v}$, we have $\int_{-h}^0 \nabla \cdot \hat{\mathbf{v}}
d \bar{z} = 0$.   Of course, \eqref{eq:evolforShiftedv} is supplemented with 
boundary conditions as in \eqref{eq:3dPEPhysicalBoundaryCondTop}--\eqref{eq:3dPEPhysicalBoundaryCondSide}.
We will hereafter prove:
\begin{Prop}\label{thm:ExistenceRegularityAuxLinHatU}
Define $(\hat{\mathbf{v}}, \hat{T}, \hat{S}) = \hat{U} := U - \check{U}$.
Then for any, for any $t > 0$, $\hat{\mathbf{v}}$ satisfies 
 \begin{equation}\label{eq:BoundednessUpToExistenceTimehat}
    \sup_{t' \in [0, t \wedge \xi]} |\hat{\mathbf{v}}(t')|_{L^4} < \infty
    \quad a.s.
  \end{equation}
\end{Prop}

\begin{proof} Note that \eqref{eq:evolforShiftedv} does not have any white noise terms. 
Multiplying  \eqref{eq:evolforShiftedv} by $\hat{\mathbf{v}}^3 = (v_{1}^{3}, v_{2}^{3})$ and then integrating
over the domain $\mathcal{M}$ leads to the following system describing the time
evolution of $|\hat{\mathbf{v}}|^4_{L^4}$:
\begin{equation}\label{eq:EvolutionvhatL4}
  \begin{split}
  \frac{1}{4} \frac{d}{dt} |\hat{\mathbf{v}}|^4_{L^4}
     &+ 3 \mu_{\mathbf{v}} \sum_{j,k}\int_{\mathcal{M}} 
                  (\pd{j} \hat{v}_k)^2 \hat{v}_k^2 \dM
     + 3 \nu_{\mathbf{v}} \sum_{k} \int_{\mathcal{M}} 
                  (\pd{z}\hat{v}_k)^2 \hat{v}_k^2  \dM\\
    =&- \sum_{j,k} \int_{\mathcal{M}} 
        (\hat{v}_j + \check{v}_j) \pd{j} \check{v}_k \hat{v}_k^3 
        \dM
        -\sum_{k} \int_{\mathcal{M}} 
       w(\hat{\mathbf{v}} + \check{\mathbf{v}}) \pd{z} \check{v}_k \hat{v}_k^3
       \dM\\
   &- \sum_{j} \frac{1}{\rho_0} \int_{\mathcal{M}} 
             \pd{j} \hat{p}_s  \hat{v}_j^3
        \dM
   + g \int_{\mathcal{M}} \int_z^0 
         \left( \beta_T  \nabla (\hat{T} + \check{T})
     + \beta_S  \nabla (\hat{S} + \check{S}) \right) d \bar{z} \cdot
       \mathbf{\hat{v}}^3 \dM\\
   &- f\int_{\mathcal{M}} 
      \mathbf{k} \times (\hat{\mathbf{v}} + \check{\mathbf{v}})
       \cdot \mathbf{\hat{v}}^3 \dM 
   + \int_{\mathcal{M}} F_{\mathbf{v}} \cdot \mathbf{\hat{v}}^3  \dM\\
   :=& J_1 + J_2 + J_3 + J_4 + J_5 + J_6.
 \end{split}
\end{equation}
Note that we have made use of the cancelation property:
\begin{equation}\label{eq:CancelationInLP}
	\int_{\mathcal{M}}(
	\mathbf{v}^{\sharp} 
	\cdot \nabla \mathbf{v} 
	\cdot \mathbf{v}^{r} +
	w(\mathbf{v}^{\sharp} )
	\pd{z} \mathbf{v} 
	\cdot \mathbf{v}^{r} )
	\dM = 0,
\end{equation}
which holds for any $r \geq 1$ and
any vector fields over $\mathcal{M}$, 
$\mathbf{v}$, $\mathbf{v}^\sharp$ that are
sufficiently smooth and satisfy the boundary
conditions  \eqref{eq:3dPEPhysicalBoundaryCondTop},
\eqref{eq:3dPEPhysicalBoundaryCondBottom},
\eqref{eq:3dPEPhysicalBoundaryCondSide}.  Notice
moreover that for some constant $\kappa>0$ (depending
only on $\mu_{\mathbf{v}}$, $\nu_{\mathbf{v}}$) we have
that
\begin{equation}\label{eq:DispObservation}
   \kappa | \nabla_{3}\hat{\mathbf{v}}^{2} |^{2}
   \leq
         \frac{\mu_{\mathbf{v}}}{2} \sum_{j,k}\int_{\mathcal{M}} 
                  (\pd{j} \hat{v}_k)^2 \hat{v}_k^2 \dM
     +  \frac{\nu_{\mathbf{v}}}{2} \sum_{k} \int_{\mathcal{M}} 
                  (\pd{z}\hat{v}_k)^2 \hat{v}_k^2  \dM,\\
\end{equation}
where $\nabla_{3} = (\pd{1}, \pd{2}, \pd{z})$ is the full
three dimensional gradient operator.\\

\noindent {\bf \emph{Estimates for the Nonlinear Terms $J_1$, $J_2$.}}

We first provide estimates for the inertial terms $J_1$, $J_2$
in \eqref{eq:EvolutionvhatL4}.
For $J_1$ we estimate:
\begin{displaymath}
  \begin{split}
  \left| \sum_{k,j} \int_{\mathcal{M}}  
  \hat{v}_j \pd{j} \check{v}_k \hat{v}_k^3 
   d\, \mathcal{M} \right|    
 \leq&  c \sum_{k,j} \int_{\mathcal{M}}  
    |\pd{j} \check{v}_k| (|\hat{v}_k|^4 + |\hat{v}_j|^4)  
    \dM 
 \leq c |\nabla \check{\mathbf{v}}|_{L^6}
          \sum_j \left(
                \int_{\mathcal{M}} \left( |\hat{v}_j|^4 
                \right) ^{6/5}\dM
             \right)^{5/6}\\
  \leq& c |A \check{U}|
     \sum_j \left(
                \int_{\mathcal{M}} \left( \hat{v}_j^2\right)^{12/5}
                \dM
             \right)^{(5/12) \cdot 2}
  \leq  c |A \check{U}|
                \left(|\hat{\mathbf{v}}^2|^{3/4}
                |\nabla_{3} (\hat{\mathbf{v}}^2)|^{1/4}\right)^2\\
  \leq& c |A \check{U}|
                |\hat{\mathbf{v}}|_{L^4}^{3}
                |\nabla_{3} (\hat{\mathbf{v}}^2)|^{1/2}
	\leq
   c |A \check{U}|^{4/3} |\hat{\mathbf{v}}|^{4}_{L^4}
             + \kappa |\nabla_{3} (\hat{\mathbf{v}}^2)|^{2}. \\
\end{split}
\end{displaymath}
Here we have used the Sobolev embeddings in $\mathbb{R}^3$ 
of $H^1$ into $L^6$ and of $H^{1/4}$ into $L^{5/12}$.  For $J_1$
we also need to estimate
\begin{displaymath}
   \begin{split}
     \left| \sum_{k,j} \int_{\mathcal{M}}  
       \check{v}_j \pd{j} \check{v}_k \hat{v}_k^3 
       d\, \mathcal{M} \right| 
     \leq c | \check{\mathbf{v}} |_{L^{\infty}} 
                  | \nabla \check{\mathbf{v}}|_{L^4}
                  |\hat{\mathbf{v}}^3|_{L^{4/3}}
      \leq c |A \check{U}|^2
                (1 + |\hat{\mathbf{v}}|_{L^{4}}^4).
  \end{split}
\end{displaymath}
In summary,
\begin{equation}\label{eq:J1EstL4}
  |J_1| \leq c (1 + |A\check{U}|^2)   (1 + |\hat{\mathbf{v}}|_{L^{4}}^4)
             +  \kappa |\nabla_{3} (\hat{\mathbf{v}}^2)|^{2}.\\
\end{equation}
The estimates for $J_2$ make use of a preliminary integration by parts
(cf. \eqref{eq:divFreeTypeCondPEInt})
\begin{displaymath}
  \begin{split}
  -\sum_{k} &\int_{\mathcal{M}} 
       w(\hat{\mathbf{v}} + \check{\mathbf{v}}) \pd{z} \check{v}_k \hat{v}_k^3
       \dM \\
      =&
  \sum_{j,k}\int_{\mathcal{M}} \int^0_z \left(
      \hat{v}_j + \check{v}_j 
      \right)d \bar{z} \,
      \pd{zj} \check{v}_k \hat{v}_k^3
      \dM 
     + 3
  \sum_{j,k} \int_{\mathcal{M}} \int^0_z\left(
     \hat{v}_j + \check{v}_j\right) d \bar{z} \,
     \pd{z} \check{v}_k 
     \pd{j} \hat{v}_k
     \hat{v}_k^2
     \dM \\
    :=& J_{2,1} + J_{2,2} +J_{2,3} + J_{2,4}.
\end{split}
\end{displaymath}
For $J_{2,1}$ we use the embedding of $H^{3/4}$ into $L^4$ in $\mathbb{R}^3$
and find,
\begin{displaymath}
  \begin{split}
  |J_{2,1}| &\leq  c |\hat{\mathbf{v}}|_{L^8}
    |\nabla \pd{z} \check{\mathbf{v}}|
    |\hat{\mathbf{v}}^3|_{L^{8/3}}
\leq c |A \check{U}| |\hat{\mathbf{v}}|_{L^8}^4
\leq c |A \check{U}| \sum_j \left(
               \int_{\mathcal{M}} \left( \hat{v}_j^2\right)^{4}
            \right)^{(1/4) \cdot 2}\\
&\leq c |A \check{U}| \left(
    |\hat{\mathbf{v}}^2|^{1/4}
    |\nabla_{3}(\hat{\mathbf{v}}^2)|^{3/4} \right)^2
\leq c |A \check{U}|^4 
    |\hat{\mathbf{v}}|_{L^4}^4
            +\frac{\kappa}{2}
    |\nabla_{3}(\hat{\mathbf{v}}^2)|^2.
\end{split}
\end{displaymath}
For $J_{2,2}$, by making use of Agmon's inequality and that $H^{1/2}$
is embedded into $L^3$ we infer:
\begin{displaymath}
  \begin{split}
    |J_{2,2}| \leq&   |\mathbf{\check{v}}|_{L^\infty} 
         |\nabla \pd{z}\mathbf{\check{v}}|
         |\mathbf{v}^3|
         \leq c
        |\mathbf{\check{v}}|_{L^\infty} 
         |\nabla \pd{z}\mathbf{\check{v}}|
         |\hat{\mathbf{v}}^2|_{L^3}^{3/2}
         \leq
         c |A \check{U}|^2 \left(
   |\hat{\mathbf{v}}^2|^{1/2}
    |\nabla_{3}(\hat{\mathbf{v}}^2)|^{1/2} \right)^{3/2}
        \leq     c |A \check{U}|^2 
   |\hat{\mathbf{v}}|_{L^4}^{3/2}
    |\nabla_{3}(\hat{\mathbf{v}}^2)|^{3/4} 
    \\
    \leq& c |A \check{U}|^{16/5} 
    |\hat{\mathbf{v}}|_{L^4}^{12/5}
        + \frac{\kappa}{2}
    |\nabla_{3}(\hat{\mathbf{v}}^2)|^2
        \leq c (1+|A \check{U}|^{4})
    (1+ |\hat{\mathbf{v}}|_{L^4}^{4})
        + \frac{\kappa}{2}
    |\nabla_{3}(\hat{\mathbf{v}}^2)|^2.
 \end{split}
\end{displaymath}
$J_{2,3}$ and $J_{2,4}$ seem to require `anisotropic' type estimates:
\begin{displaymath}
  \begin{split}
    |J_{2,3} + J_{2,4}| 
    \leq& \sum_{j,k}
   \int_{\mathcal{M}_0} 
   \left| \int^0_z\left( \hat{v}_j + \check{v}_j\right) 
     d \bar{z} \right|_{L^\infty_z}
    |\pd{z} \check{v}_k|_{L^6_z}
    |\pd{j} \hat{v}_k \hat{v}_k|_{L^2_z}
    |\hat{v}_k|_{L^3_z} \dMo\\
   \leq& \sum_{j,k}
   \int_{\mathcal{M}_0} 
     |\hat{v}_j + \check{v}_j|_{L^2_z}
   |\pd{z} \check{v}_k|_{L^6_z}
    |\pd{j} \hat{v}_k \hat{v}_k|_{L^2_z}
    |\hat{v}_k|_{L^3_z}\dMo\\
   \leq&\sum_{j,k}
     |\hat{v}_j + \check{v}_j|_{L^{12}_\mathbf{x}L^2_z}
   |\pd{z} \check{v}_k|_{L^6_\mathbf{x}L^6_z}
    |\pd{j} \hat{v}_k \hat{v}_k|_{L^2_\mathbf{x}L^2_z}
    |\hat{v}_k|_{L^4_\mathbf{x} L^3_z}\\
    \leq&c
    \|U\|
 |A \check{U}|
              \left( \sum_{j,k} \int_{\mathcal{M}} 
        (\pd{j}\hat{v}_k)^2 \hat{v}_k^2
       \dM 
  \right)^{1/2}
  |\hat{\mathbf{v}}|_{L^4}\\
   \leq& c
    \|U\|^2
 |A \check{U}|^2  (1+|\hat{\mathbf{v}}|_{L^4}^4)
         +\frac{1}{2} \mu_{\mathbf{v}}
            \sum_{j,k} \int_{\mathcal{M}} 
       (\pd{j}\hat{v}_k)^2 \hat{v}_k^2
      \dM.
 \end{split}
\end{displaymath}
Note that we made use of
Remark~\ref{rmk:embeddingLpLq}, \eqref{eq:sobEmbAniso} with $q = 12$ 
in order to estimate $|\hat{v}_j + \check{v}_j|_{L^2_zL^{12}_x}=
|v_j |_{L^2_zL^{12}_x} \leq c \|U\|$.
Summarizing the above estimates
and taking advantage of the observation 
\eqref{eq:DispObservation} we conclude that
\begin{equation}\label{eq:J2EstFinal}
  \begin{split}
  |J_2| \leq& c 
 (1 + \|U\|^2)( 1 + |A \check{U}|^4)
 (1 +|\hat{\mathbf{v}}|_{L^4}^4)
        + \mu_{\mathbf{v}}
         \sum_{j,k} \int_{\mathcal{M}} 
      (\pd{j}\hat{v}_k)^2 \hat{v}_k^2
     \dM 
     + \nu_{\mathbf{v}}
         \sum_{j,k} \int_{\mathcal{M}} 
      (\pd{z}\hat{v}_k)^2 \hat{v}_k^2
     \dM.
 \end{split}
\end{equation}


\noindent {\bf \emph{Estimates for the Pressure Term $J_3$}}

We next attend to the term $J_3$.  Using the fact that the pressure
term, $p_s$, depends only on the horizontal variable $\mathbf{x}$
we find
\begin{displaymath}
\begin{split}
     |J_3| =&   \left| \sum_{j} \frac{1}{\rho_0} \int_{\mathcal{M}_0} 
             \pd{j} \hat{p}_s  \int_{-h}^0 \hat{v}_j^3 dz
        \dMo \right|
             \leq c |\nabla p_s|_{L^{4/3}_x}
                  \left| \int_{-h}^0 \hat{\mathbf{v}}^3 dz
                 \right|_{L^4_\mathbf{x}}.\\
\end{split}
\end{displaymath}
By applying the Sobolev embedding in $D =2$ that $W^{4/3,1}$ is
embedded in $L^4(\mathcal{M}_0)$ we have
\begin{displaymath}
  \begin{split}
     \left| \int_{-h}^0 \hat{\mathbf{v}}^3 dz
                 \right|_{L^4_x}
     \leq& c  \sum_{j,k}
                \left| \pd{j} \int_{-h}^0 \hat{v}^3_k dz
               \right|_{L^{4/3}_x}
     \leq c  \sum_{j,k}
               \left| \int_{-h}^0 
                  (\pd{j} \hat{v}_k  \hat{v}_k)
                  \hat{v}_k dz
              \right|_{L^{4/3}_x}\\
     \leq& c \sum_{j,k} \left(
              \int_{\mathcal{M}_0} 
              \left( \int_{-h}^0 (\pd{j} \hat{v}_k )^2
                \hat{v}^2_k  dz \right)^{2/3}
              \left(\int_{-h}^0 \hat{v}^2_k dz \right)^{2/3} 
              \dMo
             \right)^{3/4}\\
     \leq& c 
                 |\hat{\mathbf{v}}|_{L^4}
            \sum_{j,k} \left(
            \int_{\mathcal{M}} 
             (\pd{j} \hat{v}_k )^2
              \hat{v}^2_k  \dM \right)^{1/2}.\\
\end{split}
\end{displaymath}
Combining these estimates and using Young's inequality
we conclude that:
\begin{equation}\label{eq:J3EstL4Summary}
     |J_{3}| \leq c |\nabla p_s|_{L^{4/3}_x}^2
                 |\hat{\mathbf{v}}|_{L^4}^2
                 + \mu_{\mathbf{v}}
            \sum_{j,k}
            \int_{\mathcal{M}} 
             (\pd{j} \hat{v}_k )^2
              \hat{v}^2_k  \dM.\\
\end{equation}
\begin{Rmk}\label{rmk:LowerDimPressure}
  Note that it is at precisely this stage in the estimates that we see the crucial role played
  by the two dimensional spatial dependence of the surface pressure $p_{s}$.   This insight
  concerning the importance of the lower dimensional pressure is the key 
  to the recent breakthroughs for global existence in \cite{CaoTiti,
  Kobelkov, Kobelkov2007, ZianeKukavica}.   In the present study we follow
  the approach in \cite{ZianeKukavica} which treats the pressure explicitly via
  bounds on the pressure for the Stokes equation found in 
  \cite{SohrVonWahl1} and reproduced   below in Proposition~\ref{thm:SVWResult} for the convenience of the
  reader.
\end{Rmk}

In order to estimate the term $|\nabla p_s|_{L^{4/3}}$
appearing in \eqref{eq:J3EstL4Summary} we now apply 
Proposition~\ref{thm:SVWResult} below.  To this
end we average \eqref{eq:evolforShiftedv} in the vertical
direction. Define the operator $\mathfrak{A}(\mathbf{v}) 
= \frac{1}{h} \int_{-h}^0 \mathbf{v} dz$ and
let $\mathbf{q} = \rho_0\mathfrak{A}(\hat{\mathbf{v}})$. By
applying $\mathfrak{A}$ to \eqref{eq:evolforShiftedv} we find that $\mathbf{q}$
satisfies the following Stokes system over $\mathcal{M}_0 \subset \mathbb{R}^2$:
\begin{displaymath}
  \begin{split}  
   \pd{t} \mathbf{q} - \mu_{\mathbf{v}} &\Delta \mathbf{q} + \nabla \hat{p}_s
             = G(U),\\
        \nabla \cdot \mathbf{q} &= 0,  \quad \mathbf{q}_{|\partial \mathcal{M}_0} = 0,
\end{split}
\end{displaymath}
where, for $\mathbf{v} = \hat{\mathbf{v}} + \check{\mathbf{v}}$,
\begin{equation}\label{eq:LHSforAveragingPurpose}
  \begin{split}
  G(U) &= G(\mathbf{v},T,S) = G_1(U) + G_2(U)\\
         :&= -\rho_0
     \mathfrak{A}\left( 
      (\mathbf{v} \cdot \nabla)
               \mathbf{v} 
    + (\nabla \cdot \mathbf{v})
         \mathbf{v} \right)
   + \rho_0   \mathfrak{A}\left( g \int_z^0 \left( \beta_T  \nabla T
     + \beta_S  \nabla S \right) d \bar{z}
  - f \mathbf{k} \times \mathbf{v} + F_{\mathbf{v}}
    \right).
\end{split}
\end{equation}
Note that, using \eqref{eq:divFreeTypeCondPE} 
and taking into account the boundary conditions
\eqref{eq:3dPEPhysicalBoundaryCondTop},
\eqref{eq:3dPEPhysicalBoundaryCondBottom},
$\mathfrak{A}(w(\mathbf{v}) \pd{z} \mathbf{v}) = 
   \mathfrak{A}((\nabla\cdot \mathbf{v}) \mathbf{v})$
   and
$\mathfrak{A}(\pd{zz} \hat{\mathbf{v}}) = 0$.

Making use of Proposition~\ref{thm:SVWResult} below, we find,
for any $0 < t < \infty$, and for any pair of stopping times
$\tau_a$, $\tau_b$,  with $0 \leq \tau_a \leq \tau_b 
\leq t \wedge \xi$ that
\begin{displaymath}
  \int_{\tau_a}^{\tau_b} | \nabla \hat{p}_s|^2_{L^{4/3}_\mathbf{x}}dt'
     \leq c \left( \|\mathbf{q}(\tau_a)\|^2 
        + \int_{\tau_a}^{\tau_b} |G(U) |^2_{L^{4/3}_\mathbf{x}}dt' 
        \right).
\end{displaymath}
Observe that $\|\mathbf{q}(\tau_{a})\| \leq c \|\hat{\mathbf{v}}(\tau_{a})\|$
and that
\begin{displaymath}
  \int_{\tau_a}^{\tau_b} | G_2(U)|^2_{L^{4/3}_\mathbf{x}} dt' \leq
  c \int_{\tau_a}^{\tau_b} | G_2(U)|^2 dt'
 \leq c \int_{\tau_a}^{\tau_b} (\|U\|^2 + |F|^2) dt'.
\end{displaymath}
On the other hand,
\begin{displaymath}
  \begin{split}
     \int_{\tau_a}^{\tau_b} |G_1(U)|_{L^{4/3}_\mathbf{x}}^{2}dt'
     \leq& c  \int_{\tau_a}^{\tau_b}|  (\mathbf{v} \cdot \nabla)
               \mathbf{v} 
    + (\nabla \cdot \mathbf{v})
         \mathbf{v}|_{L^{4/3}}^{2} dt'\\
          \leq& 
      c \int_{\tau_a}^{\tau_b} \left( \int_{\mathcal{M}}
        |\mathbf{v}|^{4/3} |\nabla \mathbf{v}|^{4/3}
        \dM
        \right)^{3/2}dt'
        \leq c\int_{\tau_a}^{\tau_b} |\mathbf{v}|_{L^4}^{2}\|\mathbf{v}\|^{2} dt'.
  \end{split}
\end{displaymath}
In conclusion:
\begin{equation}\label{eq:pressureEstimateConclusion}
     \int_{\tau_a}^{\tau_b} | \nabla \hat{p}_s|^2_{L^{4/3}_\mathbf{x}}dt'
     \leq c \left( \|\hat{\mathbf{v}}(\tau_a)\|^2  
        + \int_{\tau_a}^{\tau_b} ((1+|\mathbf{v}|_{L^4}^2) \|U\|^2  +
           |F|^2) dt' 
        \right).
\end{equation}

Finally, combining \eqref{eq:pressureEstimateConclusion} and 
\eqref{eq:J3EstL4Summary} we find, for any pair of stopping times
$0 \leq \tau_{a} \leq \tau_{b} \leq t \wedge \xi$,
\begin{equation}\label{eq:J3forL4FinalIntEst}
\begin{split}
\int_{\tau_a}^{\tau_b} |J_{3}| dt'
          \leq&  \frac{1}{2}\sup_{t' \in [\tau_{a}, \tau_{b}]}|\hat{\mathbf{v}}|_{L^4}^4 +
           \mu_{\mathbf{v}} \int_{\tau_a}^{\tau_b}
            \sum_{j,k}
            \int_{\mathcal{M}} 
             (\pd{j} \hat{v}_k )^2
              \hat{v}^2_k  \dM dt'\\
	 & + c \left( \|\hat{\mathbf{v}}(\tau_a)\|^4  + (1 +
	 \sup_{t' \in [\tau_{a}, \tau_{b}]}(|\hat{\mathbf{v}}|_{L^4}^4 + |\check{\mathbf{v}}|_{L^4}^4)
         \int_{\tau_a}^{\tau_b} (\|U\|^2  +
           |F|^2 ) dt' \right).
\end{split}
\end{equation}


\noindent {\bf \emph{Estimates for the Lower Order Terms $J_4$, $J_5$, $J_6$}}

To estimate the term $J_4$ we use the embedding of $L^{3}$ into $H^{1/2}$ in $2D$ and we observe
that
\begin{equation}\label{eq:J4EstL4}
\begin{split}
 |J_4| &\leq c\|U\| \sum_{k} \left( \int_{\mathcal{M}} (\hat{v}_{k}^{2})^{3} \dM \right)^{1/3 \cdot 3/2}
 	 \leq c\|U\| \left( |\hat{\mathbf{v}}^{2}|^{1/2} | \nabla_{3} \hat{\mathbf{v}}^{2}|^{1/2}\right)^{3/2}\\
	 &\leq  \kappa | \nabla_{3} \hat{\mathbf{v}}^{2}|^{2}
	              + c (1+ \|U\|^{2}) (1+| \hat{\mathbf{v}}|_{L^{4}}^{4}).
	 \end{split}
\end{equation}
The estimate for the Coriolis term $J_5$ is direct
\begin{equation}\label{eq:J5EstL4}
  |J_5| \leq c |\mathbf{v}|_{L^4} |\mathbf{\hat{v}}|_{L^4}^3
          \leq c \|U\|(1 + |\mathbf{\hat{v}}|_{L^4}^4).
\end{equation}
Finally for $J_6$ we find
\begin{equation}\label{eq:J6EstimateFL4}
  |J_6| \leq |F_{\mathbf{v}}|_{L^4}|\mathbf{\hat{v}}|^3_{L^4}\leq c
  (|F|_{L^4}^2+1)(|\mathbf{\hat{v}}|^4_{L^4} + 1).
\end{equation}

Integrating \eqref{eq:EvolutionvhatL4}
in time from $\tau_a$ to $\tau_b$
and then putting together the estimates 
\eqref{eq:DispObservation}
\eqref{eq:J1EstL4},
\eqref{eq:J2EstFinal},
\eqref{eq:J3forL4FinalIntEst},
\eqref{eq:J4EstL4},
\eqref{eq:J5EstL4},
\eqref{eq:J6EstimateFL4}
we find that for any pair of stopping times $\tau_{a}$,
$\tau_{b}$ with $0 \leq \tau_{a} \leq \tau_{b} \leq t \wedge \xi$
\begin{displaymath}
  \begin{split}
  \sup_{t' \in [\tau_{a}, \tau_{b}]}& |\hat{\mathbf{v}}|^4_{L^4}
  \leq c\|\hat{\mathbf{v}}(\tau_{a})\|^{4} + c
  \left(1 + \sup_{t' \in [\tau_{a}, \tau_{b}]} |\hat{\mathbf{v}}|^4_{L^4} 
  +\sup_{t' \in [0, t \wedge \xi]} |\check{\mathbf{v}}|^4_{L^4}
  \right)  \int_{\tau_{a}}^{\tau_b} H(t') dt',
    \end{split}
  \end{displaymath}
where
\begin{displaymath}
  H(t') :=  
  (1+\|U(t')\|^{2}) (1 + |A\check{U}(t')|^{4}) + |F(t')|^{2}_{L^4} .
\end{displaymath}
We finally apply the version of Gronwall's lemma in
Proposition~\ref{thm:GeneralizedGron} below with
$p = 2$,
$X(t') = |\hat{\mathbf{v}}(t')|^4_{L^4}$,
$f(t') = c\|\hat{\mathbf{v}}(t')\|^{2}$, $g(t') = cH(t')$
and $h(t') = c(1 +\sup_{t' \in [0, t \wedge \xi]} |\check{\mathbf{v}}|^4_{L^4}) H(t')$.
Observe that, as a consequence of Theorem~\ref{thm:LocalExistence}
and Proposition~\ref{thm:ExistenceRegularityAuxLinCheckU}
it is clear that $\hat{U} = U - \check{U} \in C([0,t \wedge \xi);V)$.
It therefore follows directly that $X \in C([0,t\wedge \xi))$ and that $f$ is continuous at $t =0$.  Also, due 
to Theorem~\ref{thm:LocalExistence},
Proposition~\ref{thm:ExistenceRegularityAuxLinCheckU} and the 
standing assumption \eqref{eq:SizeConF}, we easily infer that 
$f, g, h \in L^1([0, t \wedge \xi])$.  Thus by, Proposition~\ref{thm:GeneralizedGron},
we conclude that \eqref{eq:BoundednessUpToExistenceTimehat} holds
and Proposition~\ref{thm:ExistenceRegularityAuxLinHatU} is now proved.
\end{proof}

\begin{proof} [Proof of Proposition~\ref{thm:L4MaximalExistence} (conclusion)]
Combining \eqref{eq:BoundednessUpToExistenceTimehat}  with \eqref{eq:BoundednessUpToExistenceTimeChk}
we infer
\eqref{eq:BoundednessUpToExistenceTimeFullEqn}.  By the definition of $\tau_K^{(1)}$, \eqref{eq:BlowUpControllerTm},
it follows directly from \eqref{eq:BoundednessUpToExistenceTimeFullEqn} that
$\lim_{K \uparrow \infty} \tau_{K}^{(1)} = \infty$, this completes the proof of
Proposition~\ref{thm:L4MaximalExistence}.
\end{proof}

\section{Vertical Gradient Estimates}
\label{sec:VerticleGradEst}

In this section we will
show that $\lim_{K \uparrow \infty}\tau_{K}^{(2)} =+ \infty$,
with $\tau_{K}^{(2)}$ defined in \eqref{eq:BlowUpControllerTm}.
Recall that, in view of proving Theorem~\ref{thm:StrongNormBnds} and hence establishing
the main result of the article, we need to show that $\lim_{K \uparrow \infty}
\tau_K = + \infty$, where $\tau_K = \tau_K^{(1)} \wedge \tau_K^{(2)}$  
is given by \eqref{eq:BlowUpControllerTm}.  Since we have already shown in 
the previous section that $\lim_{K \uparrow \infty}
\tau_K^{(1)} = +\infty$  we now need to show that $\lim_{K \uparrow \infty} \tau_K^{(2)} = +\infty$
which essentially amounts to deriving suitable bounds for $\pd{z} U$ in $L^2(\mathcal{M})$.
\begin{Prop}\label{thm:VerticalGradConclusion}
  Assume the conditions set out in Theorem~\ref{thm:MainResult}
  and suppose that $(U, \xi)$ is the corresponding strong pathwise solution of
  \eqref{eq:PE3DBasic}.  Then, for any $t > 0$,
  \begin{equation}\label{eq:FinalUdzBounds}
         \sup_{t' \in [\xi \wedge t]}    |\pd{z} U|^2 +
         \int_0^{\xi \wedge t} 
     \|\pd{z} U \|^2 dt' < \infty.
   \end{equation}
   Hence, defining $\tau_K^{(2)}$ according to \eqref{eq:BlowUpControllerTm},
   $\lim_{K \uparrow \infty} \tau_K^{(2)} = \infty$.
\end{Prop}

The bound \eqref{eq:FinalUdzBounds} and hence Proposition~\ref{thm:VerticalGradConclusion} follows 
directly from \eqref{eq:NonBlowpdzv} and \eqref{eq:TemperatureBndPdzPointwise} which are
established in Corollary~\ref{thm:StoppingTmConclMomentum} and 
Proposition~\ref{thm:TempPdZEstimates} respectively below. For the analysis we make use of several additional stopping
times:
\begin{align}
   \tauz &:= \inf_{t \geq 0}
   \left\{ 
     \sup_{t ' \in [0, t \wedge \xi]} 
     |\pd{z} \mathbf{v}(t')|^2
     + \int_{0}^{t \wedge \xi}     \|\pd{z} \mathbf{v}(t')\|^2
    \geq K
     \right\}
     =\sup_{t \geq 0}
   \left\{ 
     \sup_{t ' \in [0, t \wedge \xi]} 
     |\pd{z} \mathbf{v}(t')|^2
     + \int_{0}^{t \wedge \xi}     \|\pd{z} \mathbf{v}(t')\|^2
    \leq K
     \right\}, \notag\\
    \tauT &:= \inf_{t \geq 0}
   \left\{ 
     \sup_{t ' \in [0, t \wedge \xi]} (
     |T(t')|^4_{L^4} + |S(t')|^4_{L^4} )
    \geq K
     \right\}
     =\sup_{t \geq 0}
   \left\{ 
     \sup_{t ' \in [0, t \wedge \xi]} (
     |T(t')|^4_{L^4} + |S(t')|^4_{L^4} )
    \leq K
     \right\},  \notag\\
     \tauM &:= \tauz \wedge \tauT \wedge \tau_{K}^{(1)} .
         \label{eq:StoppingTimeMomentumpdz}
\end{align}
Here $\tau_{K}^{(1)}$ is defined as in \eqref{eq:BlowUpControllerTm}.
We first tackle $\pd{z} \mathbf{v}$ in $L^2$.  From
these estimates along  with those carried out above
for $\mathbf{v}$ in $L^4(\mathcal{M})$ we infer that 
$\lim_{K \uparrow \infty} \tauz = \infty$. 
We next would like to estimate $\pd{z}T$ and $\pd{z} S$ 
in $L^2(\mathcal{M})$.  However, in order to be able to 
treat the nonlinear terms which arise (see, for example,
\eqref{eq:evolpdzTinL2} below) we first must
estimate $T$ and $S$ in $L^4(\mathcal{M})$.  
This is the reason for the second collection of 
stopping times, $\tauT$, are defined in \eqref{eq:StoppingTimeMomentumpdz}.
In this case, in contrast to the momentum
equations above, there are no pressure terms which 
means that we can perform these $L^{4}$ estimates in the original 
variables by means of the It\={o} lemma.  With
these estimates we deduce that $\lim_{K \uparrow \infty}
\tauM = \infty$ which allows us to finally establish the desired bounds for $\pd{z}T$
and $\pd{z}S$, \eqref{eq:TemperatureBndPdzPointwise}, in Proposition~\ref{thm:TempPdZEstimates}.  Having now
established \eqref{eq:FinalUdzBounds}, we thus immediately infer
that $\lim_{K \uparrow \infty} \tau_K^{(2)} = \infty$, as desired for 
Theorem~\ref{thm:StrongNormBnds}.

\subsection{Estimates for the Momentum Equations}

\begin{Prop}\label{thm:pdzvUnderControl}
  Assume the conditions set out in Theorem~\ref{thm:MainResult}
  and suppose that $(U, \xi)$ is the corresponding strong pathwise solution of
  \eqref{eq:PE3DBasic}. Define $\tau_K^{(1)}$ as in
 \eqref{eq:BlowUpControllerTm}.  Then, for any $t > 0$,
 \begin{equation}\label{eq:BoundednessInMeanMomentumConcl}
   \E \left( \sup_{t' \in [0, \tau_K^{(1)} \wedge \xi \wedge t]} 
         |\pd{z}\mathbf{v}|^2 +
         \int_0^{\tau_K^{(1)} \wedge \xi \wedge t} 
     \|\pd{z}\mathbf{v} \|^2 dt' \right) < \infty.
\end{equation}
\end{Prop}
From this proposition, proven immediately below, we obtain the following
\begin{Cor}\label{thm:StoppingTmConclMomentum}
Assume that the conditions set out in Theorem~\ref{thm:MainResult}
hold and that $(U, \xi)$ is the resulting strong, pathwise solution of \eqref{eq:PE3DBasic}.
Then for, any $t > 0$,
\begin{equation}\label{eq:NonBlowpdzv}
 \sup_{t' \in [\xi \wedge t]} 
         |\pd{z}\mathbf{v}|^2 +
         \int_0^{\xi \wedge t} 
     \|\pd{z}\mathbf{v} \|^2 dt'  < \infty
     \quad a.s.
\end{equation}
Moreover, defining the stopping times 
$\tauz$ according to \eqref{eq:StoppingTimeMomentumpdz}
we have that
 $\lim_{K \uparrow \infty}$ $\tauz = \infty.$
\end{Cor}
\begin{proof}[Proof of Proposition~\ref{thm:pdzvUnderControl}:]
We apply $\pd{z}$ to \eqref{eq:MomentumPEint} to infer an evolution
equation for $\pd{z} \mathbf{v}$:
\begin{equation}  \label{eq:evolutionv}
  \begin{split}
   d \pd{z} \mathbf{v} 
    &+ \left(\pd{z} [(\mathbf{v} \cdot \nabla)\mathbf{v} + w(\mathbf{v}) \pd{z}\mathbf{v}]
     + g\pd{z} [\beta_T  \nabla T  + \beta_S  \nabla S]
     +f \mathbf{k} \times (\pd{z} \mathbf{v})
     \right)dt\\
    &- \left( 
     \mu_{\mathbf{v}} \Delta  \pd{z} \mathbf{v} 
    + \nu_{\mathbf{v}} \pd{zzz} \mathbf{v} \right) dt
    = \pd{z} F_{\mathbf{v}} dt + \pd{z}\sigma_{\mathbf{v}}(\mathbf{v},T,S) dW.\\
  \end{split}
\end{equation}
Note that, since $\hat{p}_s$ is independent of $z$, the pressure term
does not appear above. After an application of the It\={o} formula to
this system we find
\begin{equation}\label{eq:evolutiondzL2}
  \begin{split}
    d | \pd{z} \mathbf{v}|^2 + 2\| \pd{z} \mathbf{v} \|^2dt
       =& -2\int_{\mathcal{M}}
       \left(\pd{z} [(\mathbf{v} \cdot \nabla)\mathbf{v} + w(\mathbf{v}) \pd{z}\mathbf{v}]
     + g\pd{z} [\beta_T  \nabla T  + \beta_S  \nabla S] \right)\cdot
     \pd{z}\mathbf{v}  \dM dt\\
      &+ 2(\pd{z} F_{\mathbf{v}}, \pd{z} \mathbf{v})dt
      +
       |\pd{z}
       \sigma_{\mathbf{v}}(\mathbf{v},T,S)|_{L_2(\mathfrak{U},H)}^2dt
        + 2(\pd{z}\sigma_{\mathbf{v}}(\mathbf{v},T,S), \pd{z}
       \mathbf{v}) dW\\
      =& (J_{1}^{z} + J_{2}^{z} + J_{3}^{z} + J_{4}^{z} + J_{5}^{z})dt + J_{6}^{z}dW.
\end{split}
\end{equation}
Here we have made use of the boundary conditions 
\eqref{eq:3dPEPhysicalBoundaryCondTop}, 
\eqref{eq:3dPEPhysicalBoundaryCondBottom}, 
\eqref{eq:3dPEPhysicalBoundaryCondSide} in order
to infer $(- (\mu_{\mathbf{v}} \Delta  \pd{z} \mathbf{v} 
    + \nu_{\mathbf{v}} \pd{zzz} \mathbf{v}), \pd{z} \mathbf{v})
    = \|\pd{z} \mathbf{v}\|^2$.  Note also that there is a cancelation
    in the `Coriolis term' $(k \times \pd{z} \mathbf{v},  \pd{z} \mathbf{v}) = 0$.
    
We begin by treating the nonlinear terms. For $J_1^z$, an 
integration by parts reveals
\begin{displaymath}
  \begin{split}
  J_1^z =& 
      - \sum_{j,k} \int_{\mathcal{M}} \pd{z}v_j \pd{j} v_k \pd{z}v_k 
 \dM -
   \sum_{j,k} \int_{\mathcal{M}} v_j \pd{z}\pd{j} v_k \pd{z}v_k 
 \dM\\
     =&  \sum_{j,k} \int_{\mathcal{M}} \pd{j} \pd{z}v_j  v_k \pd{z}v_k 
  \dM 
     +
       \sum_{j,k} \int_{\mathcal{M}} \pd{z}v_j v_k \pd{j}\pd{z}v_k 
 \dM 
     -
   \sum_{j,k} \int_{\mathcal{M}} v_j \pd{z}\pd{j} v_k \pd{z}v_k 
 \dM.\\
 \end{split}
\end{displaymath}
Hence with direct estimates using Holder's inequality and
the Sobolev embedding of $H^{3/4}$ into $L^4(\mathcal{M})$
\begin{equation}\label{eq:Jz1estimate}
  \begin{split}
  |J_{1}^{z}|
     \leq& c \|\pd{z} \mathbf{v}\| |\mathbf{v}|_{L^4} 
          |\pd{z} \mathbf{v}|_{L^4}
     \leq c \|\pd{z} \mathbf{v}\|^{7/4} |\mathbf{v}|_{L^4} 
          |\pd{z} \mathbf{v}|^{1/4}
     \leq c  |\mathbf{v}|_{L^4}^8 |\pd{z} \mathbf{v}|^2
          + \frac{1}{3} \| \pd{z} \mathbf{v}\|^2.
 \end{split}
\end{equation}
For the second portion, $J_2^z$, of the nonlinear term we have:
\begin{displaymath}
  \begin{split}
  J_{2}^{z}  = &  \int_{\mathcal{M}}
   w(\mathbf{v}) \pd{z}\mathbf{v}
\cdot \pd{zz}\mathbf{v} \dM
    = \frac{1}{2}  \int_{\mathcal{M}}
  w(\mathbf{v})\pd{z}\left( \pd{z}\mathbf{v}
\cdot \pd{z}\mathbf{v}\right) \dM\\
    =&  \frac{1}{2}  \int_{\mathcal{M}}
  \nabla\cdot\mathbf{v} \pd{z}\mathbf{v}
   \cdot \pd{z}\mathbf{v} \dM
    = -\sum_j \int_{\mathcal{M}}
     v_j  \pd{j}\pd{z}\mathbf{v}
   \cdot \pd{z}\mathbf{v} \dM,\\
\end{split}
\end{displaymath}
and hence, as for $J_1^z$ we may estimate
\begin{equation}\label{eq:Jz2estimate}
  \begin{split}
	|J_{2}^{z}|
     \leq& c  |\mathbf{v}|_{L^4}^8 |\pd{z} \mathbf{v}|^2
          + \frac{1}{3} \| \pd{z} \mathbf{v}\|^2.
 \end{split}
\end{equation}
Direct estimates and using the assumption \eqref{eq:lipCondSig}
lead to
\begin{equation}\label{eq:Jz345estimates}
  |J_{3}^{z}| + |J_{4}^{z}| + |J_{5}^{z}|
  	\leq c(1 + \|U\|^{2} + |F|^{2}) + \frac{1}{3}\|\pd{z} \mathbf{v}\|^{2}.
\end{equation}
On other other hand, applying the Burkholder-Davis-Gundy inequality, 
\eqref{eq:BDG}, and again using \eqref{eq:lipCondSig} we
find that for $0 \leq \tau_a \leq \tau_b \leq t \wedge \xi \wedge \tau_K^{(1)}$,
\begin{equation}\label{eq:BDGApplicationpdz}
  \begin{split}
    \E \sup_{t' \in [\tau_a, \tau_b]}&
       \left|   
        \int_{\tau_a}^{t'}  J_6^z dW
     \right|\\
      \leq& c
      \E      \left(
      \int_{\tau_a}^{\tau_b} 
      (\pd{z} \sigma_{\mathbf{v}} (U), 
      \pd{z}\mathbf{v})^2 dt'
   \right)^{1/2}
      \leq c\E      \left(
      \int_{\tau_a}^{\tau_b} 
      | \sigma_{\mathbf{v}}(U)|_{L_2(\mathfrak{U}, V)}^2 
      |\pd{z}\mathbf{v}|^2 dt'
   \right)^{1/2}\\
   \leq& 
   c\E  \left[ \sup_{t' \in [\tau_a, \tau_b]} |\pd{z}\mathbf{v}|  \left(
      \int_{\tau_a}^{\tau_b} 
      (1+ \| U \|^2)
      dt'
   \right)^{1/2} \right] 
   \leq 
   \frac{1}{2} \E  \sup_{t' \in [\tau_a, \tau_b]}
              |\pd{z}\mathbf{v}|^2 + c
               \E 
      \int_{\tau_a}^{\tau_b} (1 + \|U\|^2 )dt'.
\end{split}
\end{equation}

Gathering the estimates \eqref{eq:Jz1estimate},
\eqref{eq:Jz2estimate}, \eqref{eq:Jz345estimates},
\eqref{eq:BDGApplicationpdz} we apply the stochastic 
Gronwall Lemma, Proposition~\ref{thm:semiMGGronwall} (from \cite[Lemma 5.3]{GlattHoltzZiane2})
and infer \eqref{eq:BoundednessInMeanMomentumConcl}.
This completes the proof of Proposition~\ref{thm:pdzvUnderControl}.
\end{proof}

\subsection{$L^4(\mathcal{M})$ estimates for the Temperature and Salinity Equations}

We turn to the estimates for the temperature and
the salinity, beginning with the estimates in $L^4(\mathcal{M})$.

\begin{Lem}\label{thm:TmpBndL4}
  Suppose the standing conditions as 
  in Theorem~\ref{thm:MainResult} are satisfied and
  let $(U, \xi)$ be the maximal strong solution of
  \eqref{eq:AbstractFormulationPE}.  Define $\tau_K^W$ as in
  \eqref{eq:WeakExistence}. Then, for any $t > 0$,
  \begin{equation}\label{eq:EstimatesInL4NormT}
    \E  \left( \sup_{t' \in [0, t \wedge \tau_K^W \wedge \xi]}( |T|^4_{L^4}  + |S|^4_{L^4} )\right) < \infty,
\end{equation}
Moreover,
  \begin{equation}\label{eq:EstimatesInL4NormTPntwise}
    \sup_{t' \in [0, t \wedge \xi]}( |T|^4_{L^4}  + |S|^4_{L^4} ) < \infty.
    \quad a.s.
  \end{equation}
  With $\tauT$ defined as in 
  \eqref{eq:StoppingTimeMomentumpdz}, we have that
	$\lim_{K \uparrow \infty} \tauT = \infty$.
\end{Lem}
\begin{proof}
We carry out the estimates for the temperature $T$.  Those
for the salinity $S$ are identical.
In order to find an evolution equation for $|T|^4_{L^4}$
we may proceed similarly to e.g. \cite{Krylov1999,MikuleviciusRozovskii2001}.
We apply the It\={o} lemma to \eqref{eq:diffEqnTempPEInt}
pointwise for almost every $(\mathbf{x},z) \in \mathcal{M}$.  This is justified
given the regularity already established for $T$ in Theorem~\ref{thm:LocalExistence}.
We then integrate the resulting system over $\mathcal{M}$,
apply the stochastic Fubini theorem (see \cite{ZabczykDaPrato1}) and find that 
\begin{equation}
  \begin{split}\label{eq:EvolutionTL4}
    d|T|^4_{L^4} +& \left(
    12\mu_{T}  \sum_j \int_{\mathcal{M}} (\pd{j} T)^2 T^2\dM +
    12\nu_{T} \int_{\mathcal{M}} (\pd{z} T)^2 T^2 \dM 
      \right)dt\\
    =& 4\int_{\mathcal{M}} F_T T^3 \dM\, dt +
    6 \sum_{l \geq 1} \int_{\mathcal{M}} \sigma_T(U)_l^2 T^2 \dM dt 
   + 4 \sum_{l \geq 1} \int_{\mathcal{M}} \sigma_T(U)_l T^3 \dM dW_l.\\
  \end{split}
\end{equation}
Note also that we have once again used the cancelation properties of the
nonlinear portion of \eqref{eq:diffEqnTempPEInt} (cf. 
\eqref{eq:CancelationInLP}) in inferring \eqref{eq:EvolutionTL4}.

The external body forcing term is estimated as above (cf. \eqref{eq:J6EstimateFL4}):
\begin{equation}\label{eq:EstimatesOnExternalForceTempL4}
  \left| \int_{\mathcal{M}} F_T T^3 \dM  \right| \leq c
  (|F_T|_{L^4}^2+1)(|T|^4_{L^4} + 1).
\end{equation}
For the It\={o} correction term, using \eqref{eq:lipCondSig}, we find
\begin{equation}\label{eq:EstimateOnStochasticForceTempL4}
  \begin{split}
    \left| \sum_l \int_{\mathcal{M}} \sigma_T(U)_l^2 T^2 \dM \right|
    \leq& c\sum_l |\sigma_T(U)_l|_{L^4}^2 |T|_{L^4}^2\\
    \leq& c |\sigma_T(U)|_{L_2(\mathfrak{U}, V)}^2 |T|_{L^4}^2\\
    \leq& c (1 + \|U\|^2) (1 + |T|_{L^4}^4).\\
  \end{split}
\end{equation}
With an application of the BDG inequality, \eqref{eq:BDG}, we
find
\begin{equation}\label{eq:BDGApplcatoinTempL4}
  \begin{split}
     \E \sup_{t \in [\tau_a, \tau_b]}
    \left|
    \int_{\tau_a}^t\sum_l \int_{\mathcal{M}} \sigma_T(U)_l T^3 d
    \mathcal{M} dW_l
    \right|
 \leq&   c   \E \left(
  \int_{\tau_a}^{\tau_b} \sum_l 
  \left(
  \int_{\mathcal{M}}
    \sigma_T(U)_l T^3
         \dM 
     \right)^2
         dt
   \right)^{1/2}\\
 \leq&    c  \E \left(
  \int_{\tau_a}^{\tau_b} \sum_l 
    |\sigma_T(U)_l|_{L^4}^2 |T|_{L^4}^6
        dt
   \right)^{1/2}\\
\leq&    c  \E \left(
  \int_{\tau_a}^{\tau_b} 
      \|\sigma(U)\|^2_{L_2(\mathfrak{U}, V)} |T|_{L^4}^6
        dt
   \right)^{1/2}\\
\leq&    c  \E  \sup_{t \in [\tau_a, \tau_b]} |T|_{L^4}^3\left(
 \int_{\tau_a}^{\tau_b} (1+
     \|U\|^2) |T|_{L^4}^2
       dt
  \right)^{1/2}\\
\leq&   \frac{1}{2} \E  \sup_{t \in [\tau_a, \tau_b]} |T|_{L^4}^4 
  + c\E 
   \int_{\tau_a}^{\tau_b} (1+
     \|U\|^2)(1+|T|_{L^4}^4)
       dt.
\end{split}
\end{equation}
Here, similarly to \eqref{eq:BDGApplicationpdz} above,
$\tau_a$, $\tau_b$ may be any stopping time such that $0 \leq
\tau_a \leq \tau_b \leq t \wedge \tau_K^W \wedge \xi$.  
Combining the estimates
\eqref{eq:EstimatesOnExternalForceTempL4},
\eqref{eq:EstimateOnStochasticForceTempL4},
\eqref{eq:BDGApplcatoinTempL4}
with \eqref{eq:EvolutionTL4} we apply the stochastic Gronwall lemma, Proposition~\ref{thm:semiMGGronwall},
to infer \eqref{eq:EstimatesInL4NormT}.   Now,
due to Theorem~\ref{thm:LocalExistence}, $\lim_{K \uparrow \infty} \tau_{K}^{W} = \infty$.
With this observation \eqref{eq:EstimatesInL4NormTPntwise} 
and thus that $\lim_{K \uparrow \infty} \tauT = \infty$ follow directly 
from \eqref{eq:EstimatesInL4NormT}.  This completes the proof of Lemma~\ref{thm:TmpBndL4}
\end{proof}

The final step is now to carry out the estimates 
for $\pd{z}T$ and $\pd{z} S$ in $L^2(\mathcal{M})$:
\begin{Prop}\label{thm:TempPdZEstimates}
  Assume the conditions of Theorem~\ref{thm:MainResult}
  consider $(U, \xi)$ the maximal strong solution of
  \eqref{eq:AbstractFormulationPE}.  Take $\tauM$
  according to \eqref{eq:StoppingTimeMomentumpdz}.
  Then, for any $t > 0$,
  \begin{equation}\label{eq:TemperatureBndPdz}
       \E \left( \sup_{t \in [0,t \wedge \tauM \wedge \xi]} 
       (|\pd{z} T|^2 + |\pd{z} S|^2) +
       \int_0^{t \wedge \tauM \wedge \xi} 
       (\|\pd{z}T\|^2 + \|\pd{z}S\|^2)dt'
       \right) < \infty.
  \end{equation}
  Moreover, up to a set of measure zero,
 \begin{equation} \label{eq:TemperatureBndPdzPointwise}
    \sup_{t \in [0,t \wedge \xi]} 
       (|\pd{z} T|^2 + |\pd{z} S|^2) +
       \int_0^{t \wedge \xi} 
       (\|\pd{z}T\|^2 + \|\pd{z}S\|^2)dt'
        < \infty.
  \end{equation}
\end{Prop}

\begin{proof}
As above, in Lemma~\ref{thm:TmpBndL4} we provide 
the estimates for $T$; those for $S$ are identical. To 
this end we apply
$\pd{z}$ to \eqref{eq:diffEqnTempPEInt} and find:
\begin{equation}\label{eq:pdzTempEquation}
  \begin{split}
      d \pd{z} T + \pd{z} [(\mathbf{v}\cdot \nabla) T
           + w(\mathbf{v}) \pd{z} T]dt
           - (\mu_{T} \Delta \pd{z} T
            +\nu_{T} \pd{zz} \pd{z} T) dt
         = \pd{z} F_{T}dt + \pd{z}\sigma_{T}(\mathbf{v},T,S) dW.
 \end{split}
\end{equation}
An application of the It\={o} formula then reveals:
\begin{equation}\label{eq:evolpdzTinL2}
  \begin{split}
    d|\pd{z} T|^2 + 2 \| \pd{z} T \|^2 dt =&
    2(\pd{z} F_{T}, \pd{z} T)dt +
    |\pd{z}\sigma_{T}(\mathbf{v},T,S)|^2_{L_2(\mathfrak{U}, H)}dt
    \\
       & - 2 \int_{\mathcal{M}} 
     \pd{z} [(\mathbf{v}\cdot \nabla) T
           + w(\mathbf{v}) \pd{z} T] \pd{z}T
            \, \dM dt
         + 2(\pd{z}\sigma_{T}(\mathbf{v},T,S),\pd{z} T) dW\\
     :=& (J_1^T + J_2^T + J_3^T + J_4^T)dt + J_5^TdW\\
\end{split}
\end{equation}
For the nonlinear term, $J_{3}^{T}$, an integration by parts reveals that,
\begin{displaymath}
  \begin{split}
      -  J_3^T  &= \sum_j
  \int_{\mathcal{M}} (v_j \pd{jz} T
         + \pd{z}v_j \pd{j} T) \pd{z}T
        \, \dM
          = \sum_j
  \int_{\mathcal{M}} (
   v_j \pd{jz} T \pd{z}T
  -\pd{zj}v_j  T  \pd{z}T
   - \pd{z}v_j  T \pd{zj}T)
        \, \dM,\\ 
 \end{split}
\end{displaymath}
and we may therefore estimate
\begin{equation}\label{eq:QEasyNonlinearPartTpdz}
  \begin{split}
       |J_3^T|  \leq& c(
       |\mathbf{v}|_{L^4} \|\pd{z}T\| |\pd{z}T|_{L^4}
    +  \|\pd{z}\mathbf{v}\| |T|_{L^4} |\pd{z}T|_{L^4}
+ |\pd{z}\mathbf{v}|_{L^4} |T|_{L^4} \|\pd{z}T\|)\\
       \leq& c(
       |\mathbf{v}|_{L^4}\|\pd{z}T\|^{7/4} |\pd{z}T|^{1/4} +
      \|\pd{z}\mathbf{v}\| |T|_{L^4} |\pd{z}T|^{1/4}\|\pd{z}T\|^{3/4} 
     |\pd{z}\mathbf{v}|^{1/4}\|\pd{z}\mathbf{v}\|^{3/4}|T|_{L^4}
    \|\pd{z}T\|)\\
      \leq& c(
      |\mathbf{v}|_{L^4}^8 |\pd{z}T|^2 +
     \|\pd{z}\mathbf{v}\|^2 |T|_{L^4}^2
     + \|\pd{z}\mathbf{v}\|^{2} +|\pd{z}\mathbf{v}|^{2}|T|_{L^4}^8 )
   +\frac{1}{2}(|\pd{z}T|^2 + \|\pd{z} T\|^2).
\end{split}
\end{equation}
Also,
\begin{displaymath}
  \begin{split}
       J_4^T  &= 
   \int_{\mathcal{M}}  
         w(\mathbf{v}) \pd{z} T\pd{zz}T
         \, \dM
                    = -\frac{1}{2}
  \int_{\mathcal{M}}  
       \pd{z}  w(\mathbf{v}) (\pd{z} T)^2
        \, \dM
        = \frac{1}{2}
 \int_{\mathcal{M}}  
      \pd{j} v_j (\pd{z} T)^2
       \, \dM
        = -
 \int_{\mathcal{M}}  
       v_j \pd{jz} T\pd{z}T
       \, \dM,\\
  \end{split}
\end{displaymath}
so that
\begin{equation}\label{eq:QHardNonlinearPartTpdz}
  \begin{split}
      |J_4^T| \leq c |\mathbf{v}|_{L^4} \|\pd{z}T\|^{7/4} |\pd{z}T|^{1/4}
       \leq c |\mathbf{v}|_{L^4}^8 |\pd{z}T|^2 + \frac{1}{2} \|\pd{z}T\|^2.
\end{split}
\end{equation}
The remaining drift terms are estimated directly using assumption
\eqref{eq:lipCondSig}:
\begin{equation}\label{eq:auxdriftTemppdz}
  |J_1^T+ J_2^T| \leq c( 1 + \|U\|^2 + |F|^2)  + \|\pd{z}T\|^2.
\end{equation}
As above, in \eqref{eq:BDGApplicationpdz}, we estimate the terms 
involving $J_5^T$ with the BDG
inequality. We obtain for $\tau_a \leq \tau_b \leq t \wedge \tauM
\wedge \xi$
\begin{equation}\label{eq:BDGTmppdz}
 \E \sup_{t \in [\tau_a, \tau_b]} 
   \left| \int_{\tau_a}^t J_5 dW \right|     
   \leq  \frac{1}{2}
     \E \sup_{t \in [\tau_a, \tau_b]}  |\pd{z}T|^2
     + c\E \int_{\tau_a}^{\tau_b}(1 + \|U\|^2)dt'.
\end{equation}
With the stochastic Gronwall lemma, Proposition~\ref{thm:semiMGGronwall},
the estimates \eqref{eq:QEasyNonlinearPartTpdz},
\eqref{eq:QHardNonlinearPartTpdz}, 
\eqref{eq:auxdriftTemppdz},
\eqref{eq:BDGTmppdz} imply
\eqref{eq:TemperatureBndPdz}. 
\eqref{eq:TemperatureBndPdzPointwise}
now follows directly from \eqref{eq:TemperatureBndPdz},
Proposition~\ref{thm:L4MaximalExistence}, 
Corollary~\ref{thm:StoppingTmConclMomentum}
and Lemma~\ref{thm:TmpBndL4}.
The proof of Proposition~\ref{thm:TempPdZEstimates} is thus complete.
\end{proof}

\begin{proof} [Proof of Proposition~\ref{thm:VerticalGradConclusion} (conclusion)]
Due to Corollary~\ref{thm:StoppingTmConclMomentum} and 
Proposition~\ref{thm:TempPdZEstimates} we infer 
\eqref{eq:NonBlowpdzv} and \eqref{eq:TemperatureBndPdzPointwise} 
and hence \eqref{eq:FinalUdzBounds}.  Now \eqref{eq:FinalUdzBounds}
implies that $\lim_{K \uparrow \infty}\tau_K^{(2)} = \infty$, completing the proof of the proposition.
\end{proof}

\section{Appendix: Auxilary Results}
\label{sec:auxilary-results}

We collect here, for the convenience of the reader, several
technical results that have been used in an essential way in the
analysis above.

\subsection{Pressure Estimates}
\label{sec:PressEst}

We have made use of a special case of \cite[Theorem 2.12]{SohrVonWahl1} 
in the proof of Proposition~\ref{thm:ExistenceRegularityAuxLinHatU}.
This result in \cite{SohrVonWahl1} provides $L^s_tL^r_x$ estimates for the pressure 
terms appearing in the linear Stokes equation.
\begin{Prop}\label{thm:SVWResult}
  Suppose that $d \geq 2$ and that $\mathcal{M}_0 \subset \mathbb{R}^d$ 
  is a bounded open domain with smooth boundary.  Assume that
  $r \in (1,2)$ and that $\mathbf{f}$ 
  $\in$ $L^2_{loc}([0, \infty);$ $ L^r(\mathcal{M}_0)).$
  If $(\mathbf{q}, p)$ solves the Stokes equation in $\mathcal{M}_0 \times
  [0,\infty)$
  \begin{gather*}
     \pd{t} \mathbf{q} - \nu \Delta \mathbf{q} + \nabla p = \mathbf{f}, \quad
          \nabla \cdot \mathbf{q} = 0, \quad 
          \mathbf{q}_{| \partial \mathcal{M}_0} = 0,
 \end{gather*}
 then, for any $0 \leq \tau_0 \leq \tau_1 < \infty$,
 \begin{equation}\label{eq:PressureEstimate}
   \int_{\tau_0}^{\tau_1} | \nabla p|^2_{L^r(\mathcal{M}_0)} ds
      \leq  c \left( \|\mathbf{q}(\tau_0)\|^2 + \int_{\tau_0}^{\tau_1}
      |\mathbf{f}|_{L^r(\mathcal{M}_0)}^2 ds \right),
 \end{equation}
 where $c = c(d,r, \mathcal{M}_0)$ is an absolute constant 
 independent of $\mathbf{f}$, $\tau_0$, $\tau_1$.
\end{Prop}

\subsection{Two Nonstandard Gronwall Lemmas}

We made significant use of two non-standard versions of the 
Gronwall Lemma in the analysis above.
The first version, Proposition~\ref{thm:GeneralizedGron}, is used 
to close the estimates for $\mathbf{v}$ in $L^4$ for the proof of 
Proposition~\ref{thm:L4MaximalExistence} in
Section~\ref{sec:L4Est}.  The second version, 
Proposition~\ref{thm:semiMGGronwall}, has been useful for the 
analysis of stochastic evolution equation and is employed here in
Theorem~\ref{thm:StrongNormBnds},  Proposition~\ref{thm:pdzvUnderControl}, Lemma~\ref{thm:TmpBndL4},
and Proposition~\ref{thm:TempPdZEstimates}.

The following proposition is somehow new:
\begin{Prop}\label{thm:GeneralizedGron}
  Suppose that, for some $t \geq 0$, we are given $f, g, h \in
  L^1([0,t])$ and $X \in C([0,t))$.  Assume that 
  $X, f, g, h$ are all positive for a.e. $t' \in [0,t]$ and that 
  $f$ is continuous at $t' = 0$.
  If, for almost every $0 \leq \tau_a \leq \tau_b < t$, 
  and  for some fixed $p\geq 1$, we have
  \begin{equation}\label{eq:sufficentCondForGron}
    \sup_{t' \in [\tau_a, \tau_b]} X(t') 
     \leq f(\tau_a)^p + \sup_{t' \in [\tau_a, \tau_b]} X(t') 
    \int_{\tau_a}^{\tau_b} g(t')dt' +
   \int_{\tau_a}^{\tau_b} h(t') dt',
  \end{equation}
	then there exists an absolute constant $c$
  $=c$$(t, p, $ $|f(0)|,$ $|f|_{L^1},$ $|g|_{L^1})$ such that,
  \begin{equation}  \label{eq:GronwallConclusion}
    \sup_{t' \in [0,t]} X(t') \leq c\left( 1 + \int_0^t h(t') dt' \right).
  \end{equation}
\end{Prop}

\begin{proof}
  We begin with some preliminaries to determine the constant $c$ in 
  \eqref{eq:GronwallConclusion}.
  Choose $\epsilon >0$ so that $\int_{s'}^{s} g(t') dt' < 1/2$ for any pair
  $0 \leq s' \leq s \leq t$ such that $s - s' < \epsilon$.  Now choose
  $n$ in such a way that $t/n < \epsilon/2$.  Then pick $M > 
  \max\{2, |f(0)|\}$ large enough such that $\lambda(|f| \geq M) <
  t/(4n)$, where $\lambda$ is the Lebesgue measure on the interval
  $[0,t]$.  This later quantity $M$ exists as a consequence of the
  Chebshev inequality and the fact that $f \in L^1([0,t])$.   We now show
  \eqref{eq:GronwallConclusion} is satisfied by taking $c = 2 M^p$.
  
  To this end fix any $t^* \in (t/2,t)$.  Set $t_0 = 0$, $t_{n} = t^{*}$.  
  For each $k = 1, \ldots, n-1$ we may pick an element $t_k \in
  (t^{*}k/n, t^{*}(k+1)/n)$ so that $f(t_k) < M$.   Note also that 
  according to this choice $t_{k+1} - t_{k} < \epsilon$.   Thus, for 
  $k = 0,1, \ldots, n$, by applying  \eqref{eq:sufficentCondForGron} with 
  $\tau_a = t_k$, $\tau_b= t_{k+1}$, we infer,
  \begin{displaymath}
        \sup_{t' \in [t_{k}, t_{k+1}]} X(t')
        \leq M^p + \frac{1}{2}\sup_{t' \in [t_k, t_{k+1}]} X(t') +
        \int_{t_k}^{t_{k+1}} h(t') dt'.
  \end{displaymath}
  Noting that $X \in C([0,t)$, we may rearrange and infer:
  \begin{equation}\label{eq:GronAbsStep}
    \begin{split}
      \sup_{t' \in [t_{k}, t_{k+1}]} X(t')
      \leq 2M^p + 2
        \int_{t_k}^{t_{k+1}} h(t') dt'
        \leq 2M^p \left(1 +
        \int_{t_k}^{t_{k+1}} h(t') dt'\right).
    \end{split}
  \end{equation}
  The analogue of \eqref{eq:GronwallConclusion}, with $t^{*}$ replacing $t$ on the left hand
  side of the inequality now follows from a simple induction.  Since $M$ maybe
  chosen independently of $t^{*} \in (t/2,t)$ the proof is now complete.
\end{proof}

The following `stochastic Gronwall Lemma' has been established
in \cite{GlattHoltzZiane2}:
\begin{Prop}\label{thm:semiMGGronwall}
  Fix $t > 0$ and assume that
    $X,Y,Z,R : [0,t) \times \Omega \rightarrow \mathbb{R}$    
  are real valued, non-negative stochastic processes.  Let $\tau < t$ be
  a stopping time so that
  \begin{equation}\label{eq:finiteCond}
    \mathbb{E} \int_0^{\tau} (R X + Z )\, ds < \infty.
  \end{equation}
  Assume, moreover that for some fixed constant $\kappa$
  \begin{equation}\label{eq:asBnddnessA}
    \int_0^\tau R \,ds < \kappa, \quad  a.s.
  \end{equation}
  Suppose that for all stopping times $0 \leq \tau_a < \tau_b \leq \tau$
  \begin{equation}\label{eq:basicInequality}
    \mathbb{E}\left(
      \sup_{t \in [\tau_a,\tau_b]} X + \int_{\tau_a}^{\tau_b} Y\,  ds
      \right)
      \leq c_0 \mathbb{E}\left( X(\tau_a) + \int_{\tau_a}^{\tau_b} (R X + Z )\,  ds \right),
  \end{equation}
  where $c_0$ is a constant independent of the choice of
  $\tau_a,\tau_b$.  Then
  \begin{equation}\label{eq:gronwallConc}
    \mathbb{E}\left(
      \sup_{t \in [0,\tau]} X + \int_{0}^{\tau} Y \, ds
      \right)
      \leq c \mathbb{E}\left( X(0) + \int_{0}^{\tau} Z \,ds \right),
  \end{equation}
  where $c = c(c_0,t, \kappa)$.
\end{Prop}

\section*{Acknowledgments}
This work was partially supported by the National Science Foundation (NSF)
under the grants DMS-1004638, DMS-0906440 and by the Research Fund of 
Indiana University.  We thank M. Petcu for her helpful comments and suggestions.

\bibliographystyle{alpha}
\bibliography{/Users/Nathan/Desktop/Work/Research/Reference/ref}

\newpage

\noindent Arnaud Debussche\\ {\footnotesize
D\'{e}partement de Math\'{e}matiques \\
ENS Cachan Bretagne\\
Web: \url{http://w3.bretagne.ens-cachan.fr/math/people/arnaud.debussche/}\\
Email: \url{arnaud.debussche@bretagne.ens-cachan.fr} } \\[.3cm]
Nathan Glatt-Holtz\\ {\footnotesize
Department of Mathematics\\
Indiana University\\
Web: \url{http://mypage.iu.edu/\~negh/}\\
 Email: \url{negh@indiana.edu}} \\[.3cm]
Roger Temam\\ {\footnotesize
Department of Mathematics\\
Indiana University\\
Web: \url{http://mypage.iu.edu/~temam/}\\
 Email: \url{temam@indiana.edu}}\\[.3cm]
Mohammed Ziane\\ {\footnotesize
Department of Mathematics\\
University of Southern California\\
Web: \url{http://www-bcf.usc.edu/~ziane/}\\
 Email: \url{ziane@usc.edu}}\\

\end{document}